\documentclass[12pt,reqno]{amsart}

\usepackage[T1]{fontenc}
\usepackage[utf8]{inputenc}
\usepackage{amssymb,latexsym,amsthm,amsfonts,amscd,pgf,enumerate,ragged2e}
\usepackage{amsmath}
\usepackage{pstricks-add}
\usepackage{graphics,pstricks}
\usepackage{pgf,tikz}
\usepackage{mathrsfs}
\usepackage{fullpage}
\usepackage{microtype,textcomp,fbb}
\usepackage{pst-node}
\usepackage{tikz-cd}

\usetikzlibrary{arrows}
\usepackage[pdftex, colorlinks, linkcolor=red,citecolor=blue]{hyperref}
\usepackage{cleveref}

\newtheorem{proposition}{Proposition}[section]
\newtheorem{lemma}[proposition]{Lemma}
\newtheorem{theorem}[proposition]{Theorem}
\newtheorem{corollary}[proposition]{Corollary}

\theoremstyle{definition}
\newtheorem{remark}[proposition]{Remark}

\newtheorem{example}[proposition]{Example}
\newtheorem{definition}[proposition]{Definition}

\tikzstyle{place}=[draw,circle,minimum size=1mm,inner sep=1pt,outer sep=-1.1pt,fill=black]

\usetikzlibrary{shapes.geometric}
\tikzstyle{places}=[draw,rectangle,minimum size=8pt,inner sep=0pt]
\tikzstyle{placesf}=[draw,rectangle,minimum size=5pt,inner sep=0pt]
\tikzstyle{placec}=[draw,circle,minimum size=8pt,inner sep=0pt]
\tikzstyle{placecf}=[draw,circle, minimum size=5pt,inner sep=0pt]

\newcommand{\rindj}{\mathrm{Ind}_r(G_1*G_2)}

\newcommand{\K}{\mathbb K}
\newcommand{\N}{\mathbb N}

\newcommand{\intersectionw}{(\Gamma_1\cap\Gamma_2)[W]}

\def\H{\mathcal{H}}
\def\V{\mathcal{V}}
\def\E{\mathcal{E}}
\def\C{\mathrm{Con}_r(G)}

\begin{document}


\title[Hyperedge ideal of join of graphs]{Graded Betti numbers of a hyperedge ideal associated to join of graphs}
\thanks{Last updated: \today}
\thanks{Corresponding author: Amit Roy}

\author{Sourav Kanti Patra}
\address{Department of Mathematics, Kishori Sinha Mahila College (Constituent unit of Magadh University, Bodh Gaya), Aurangabad, Bihar-824101.}
\email{souravkantipatra@gmail.com}

\author{Amit Roy}
\address{Chennai Mathematical Institute, H1 Sipcot IT Park, Siruseri, Kelambakkam, Chennai 603103, India}
\email{amitiisermohali493@gmail.com}


\begin{abstract}
Let $G$ be a finite simple graph on the vertex set $V(G)$ and let $r$ be a positive integer. We consider the hypergraph $\mathrm{Con}_r(G)$ whose vertices are the vertices of $G$ and the (hyper)edges are all $A\subseteq V(G)$ such that $|A|=r+1$ and the induced subgraph $G[A]$ is connected. The (hyper)edge ideal $I_r(G)$ of $\C$ is also the Stanley-Reisner ideal of a generalization of the independence complex of $G$, called the $r$-independence complex $\mathrm{Ind}_r(G)$. In this article, we make extensive use of the Mayer-Vietoris sequence to find the graded Betti numbers of $I_r(G_1*G_2)$ in terms of the graded Betti numbers of $I_r(G_1)$ and $I_r(G_2)$, where $G_1*G_2$ is the join of $G_1$ and $G_2$. Moreover, we find formulas for the graded Betti numbers and the Castelnuovo-Mumford regularity of $I_r(G)$, when $G$ is a complete graph, complete multipartite graph, cycle graph, and the wheel graph.
\end{abstract}

\maketitle


\section{Introduction}

A hypergraph $\H=(\V,\E)$ on the vertex set $\V$ is a collection of subsets $\E$ of $\V$, satisfying the following two conditions: (i) $|E|\ge 2$ for all $E\in \E$, and (ii) $E_i\nsubseteq E_j$ for all $E_i,E_j\in\E$. The set $\E$ is called the {\it hyperedge set} of $\H$. We sometimes write $\V(\H)$ and $\E(\H)$ in place of $\V$ and $\E$, respectively, in order to emphasize the hypergraph $\H$. Given a hypergraph $\H$ on the vertex set $\V=\{x_1,\ldots,x_n\}$ one can associate it with a square-free monomial ideal in the polynomial ring $R=\K[x_1,\ldots,x_n]$, where $\K$ is a field. The association is given by defining the monomial ideal $I(\H)=\langle \mathbf{x}_A:=\prod_{x_i\in E}x_i\mid E\in \E\rangle$ in $R$. The ideal $I(\H)$ is called the {\it hyperedge ideal}, or simply the edge ideal of $\H$, and much of the algebraic and homological properties of the ideal are governed by the combinatorics of $\H$. 

An active area of research in combinatorial commutative algebra is to study the homological properties of $I(\H)$ in terms of the combinatorial properties of $\H$. For example, in \cite{HVT} H\'a and Van Tuyl explicitly computed the Castelnuovo-Mumford regularity of $I(\H)$ when $\H$ is a triangulated hypergraph. In \cite{Emtander} Emtander defined a certain types of complete hypergraph and complete multipartite hypergraph and computed the Betti numbers and the regularity of the corresponding edge ideals. Emtander et al. in \cite{EFMM} have determined the Poincar\'e series of some hypergraph algebras generalizing lines, cycles, and stars.

One of the motivations for studying edge ideals of hypergraphs comes from its connection with a celebrated theorem by Fr\"oberg in the context of edge ideals of graphs. Let $G$ be a finite simple graph with vertex set $V(G)$ and edge set $E(G)$. If $V(G)=\{x_1,\ldots,x_n\}$ then we associate $G$ with a quadratic square-free monomial ideal $I(G)=\langle 
x_ix_j\mid \{x_i,x_j\}\in E(G) \rangle$ in the polynomial ring $\mathbb K[x_1,\ldots,x_n]$. The ideal $I(G)$ is called the {\it edge ideal} of $G$. In 1988 Fr\"oberg \cite{RF} showed that a square-free monomial ideal of degree $2$ has linear resolution if and only if it is the edge ideal of the complement of a chordal graph; thus providing a combinatorial characterization of such ideals. Since then researchers have been trying to find an analogue of this result for square-free monomial ideals in higher degrees. However, for such ideals generated in degree $3$ the resolution of the ideal depends on the characteristic of the base field (see \cite[Section 4]{MK}). Because of this one can ask for a combinatorial characterisation for those square-free monomial ideals in higher degrees which have linear resolution over all field $\K$. One way to answer this question is by defining a hypergraph analog of chordal graphs and then considering their hyperedge ideals. In this direction, Van Tuyl and Villarreal \cite{VTV}, Woodroofe \cite{Woodroofe}, Emtander \cite{Emtander1}, Bigdeli, Yazdan Pour and Zaare-Nahandi \cite{BiYPZN} independently gave different notions of chordal hypergraphs and showed that the hyperedge ideals of complement of such hypergraps have linear free resolution over any field. However, there are examples of square-free monomial ideals having linear free resolution over any field $\K$ but not being a hyperedge ideal of the complement of a chordal hypergraph in any of the above senses (see \cite[Section 4]{BiYPZ}). Thus the quest for finding such a characterisation is still wide open. 

In this paper, we study the edge ideal $I_r(G)$ of a family of hypergraphs $\C$ induced from a given graph $G$. Let $r$ be a positive integer and let $G$ be a finite simple graph on the vertex set $V(G)$ and edge set $E(G)$. Define the hypergraph $\mathrm{Con}_r(G)$ as follows.
\begin{align*}
    \V(\mathrm{Con}_r(G))&=V(G),\\
    \E(\mathrm{Con}_r(G))&=\{A\subseteq V(G)\mid |A|=r+1\text{ and }G[A]\text{ is connected }\}.
\end{align*}
Note that $\mathrm{Con}_1(G)=G$ and the ideal $I_1(G)$ is the well-known edge ideal of $G$. If $V(G)=\{x_1,\ldots,x_n\}$ then we denote the polynomial ring $\mathbb K[x_1,\ldots,x_n]$ as $R_G$. In \cite{ADGRS} it was shown that if $G$ is a tree then $\C$ is a chordal hypergraph in the sense of Woodroofe \cite{Woodroofe}. Thus the edge ideal of the complement of $\C$ has linear resolution when $G$ is a tree. The independence complex of $\C$ is called the $r$-independence complex of $G$ and is denoted by $\mathrm{Ind}_r(G)$. The complex $\mathrm{Ind}_r(G)$ can be realized as a generalization of the independence complex $\mathrm{Ind}(G)$ of $G$. Analogous to $\mathrm{Ind}(G)$, the independence complex of $\C$ also has some interesting topological and combinatorial properties. For example, in \cite{PS}, Paolini and Salvetti established a relationship between certain twisted (co)homology of the classical Braid groups and the (co)homology of the independence complexes of $\C$ of certain graphs $G$. In \cite{DS} the authors have shown that the independence complex of $\C$ is homotopy equivalent to a wedge of spheres when $G$ is a cycle graph or a perfect m-ary tree. Extending a result by Meshulam \cite{Meshulam}, Deshpande, Shukla, and Singh showed in \cite{DSS} that the (homological) connectivity of the independence complex of $\C$ gives an upper bound for the distance $r$-domination number of the underlying graph $G$. In the same paper, the authors also proved that the independence complex of $\C$, when $G$ is a chordal graph, is homotopy equivalent to a wedge of spheres for every $r\ge 1$. 
From the perspective of the Cohen-Macaulay conditions it was shown in \cite{ADGRS} that the independence complex of $\C$ is shellable when $G$ is a tree.  Moreover, it was proved using commutative algebra techniques that these complexes are vertex decomposable, the strongest among the Cohen-Macaulay conditions when the underlying graph is a caterpillar graph.

The edge ideal $I_r(G)$ of $\C$ also has some interesting algebraic and homological properties. For example, it was shown in \cite{DRST} that if the complement of $G$ is chordal then $I_r(G)$ has linear resolution; thus providing a partial generalization of Fr\"oberg's theorem. In fact, the authors have proved that the ideal $I_r(G)$ is vertex splittable, a stronger property than having linear resolution. Using the Betti splitting of vertex splittable ideal they also provided explicit formulas for $\N$-graded Betti numbers of $I_r(G)$ for some well-known families of graphs.

The graded Betti numbers and the Castelnuovo-Mumford regularity are two important numerical invariants of the minimal free resolution of a module. One current topic of research in this area is to determine the Betti numbers and the regularity of the ideal $I(\H)$ for various classes of $\H$ (see, for example, \cite{Emtander,LM, MKM}). In this paper, we focus on the Betti numbers of the hyperedge ideal $I_r(G)$, where the graph $G$ can be written as a join of two graphs $G_1$ and $G_2$. Note that if $V(G_1)$ and $V(G_2)$ are two disjoint sets then the join $G_1*G_2$ of $G_1$ and $G_2$ is defined as follows:
\begin{align*}
    V(G_1*G_2)&=V(G_1)\sqcup V(G_2),\\
    E(G_1*G_2)&=E(G_1)\cup E(G_2)\cup\{\{x,y\}\mid x\in V(G_1)\text{ and }y\in V(G_2)\}.
\end{align*}
Let $G=G_1*G_2$. For $r=1$, Mousivand \cite{AM} expressed the graded Betti numbers of $R_{G}/I_r(G)$ in terms of the graded Betti numbers of $R_{G_1}/I_r(G_1)$ and $R_{G_2}/I(G_2)$. In this article, we find such formulas for a general $r$. Indeed, we prove the following theorem which relates the Betti numbers of $R_{G_1*G_2}/I_r(G_1*G_2)$ with the Betti numbers of $R_{G_1}/I_r(G_1)$ and $R_{G_2}/I_r(G_2)$.

\begin{theorem}[\Cref{main theorem}]
    Let $G_1$ and $G_2$ be two graphs on the vertex sets $\{x_1,\ldots,x_n\}$ and $\{y_1,\ldots,y_m\}$, respectively. Let $r$ be a positive integer. Then the $\mathbb N$-graded Betti numbers of $R_{G_1*G_2}/I_r(G_1*G_2)$ can be expressed as follows. {\footnotesize
    \begin{align*}
    &\beta_{i,i+d}(R_{G_1*G_2}/I_r(G_1*G_2))\\
    &=\begin{cases}
        \sum_{j=0}^{i+d-1}\left\{{m\choose j}\beta_{i-j,i-j+d}(R_{G_1}/I_r(G_1))+{n\choose j}\beta_{i-j,i-j+d}(R_{G_2}/I_r(G_2)) \right\}+\\
\hspace{11em}   \sum_{j=1}^{i+d-1} \left[{i+d-1\choose d}-{i+d-1-j\choose d}-{j-1\choose d}\right]{m\choose j}{n\choose i+d-j}\hspace{2em}\text{ for }d=r,
 \\
    \sum_{j=0}^{i+d-1}\left\{{m\choose j}\beta_{i-j,i-j+d}(R_{G_1}/I_r(G_1))+{n\choose j}\beta_{i-j,i-j+d}(R_{G_2}/I_r(G_2)) \right\}\hspace{4.7em}\text{ for }d\ge r+1.
    \end{cases}
    \end{align*}}
    \label{intro main theorem}
\end{theorem}
Using the above theorem we can prove the following result on the (Castelnuovo-Mumford) regularity of $R_{G_1*G_2}/I_r(G_1*G_2)$.

\begin{proposition}[\Cref{main corollary}]
     Let $G_1$ and $G_2$ be two graphs on disjoint vertex sets and let $r$ be a positive integer. Then 
    \[
    \mathrm{reg}(R_{G_1*G_2}/I_r(G_1*G_2))=\max\left\{\mathrm{reg}(R_{G_1}/I_r(G_1)),\mathrm{reg}(R_{G_2}/I_r(G_2))\right\}.
    \]
\end{proposition}
Moreover, as an application of \Cref{intro main theorem} we provide explicit formulas for the Betti numbers of $I_r(G)$ when $G$ is a complete graph (\Cref{complete graph}) and complete multipartite graph (\Cref{bipartite}). We also relate the ideal $I_r(G)$ with the $t$-path ideal of $G$, when $G$ is a cycle graph (see \Cref{cycle graph}). Using this we provide formulas for the graded Betti numbers of $I_r(W_{n+1})$, where $W_{n+1}$ is the wheel graph on $n+1$ vertices (\Cref{wheel graph}). Moreover, formulas for the Castelnuovo-Mumford regularity of these ideals are obtained in \Cref{corollary 1}, \Cref{corollary 2} and \Cref{corollary 3}.

The paper is organized as follows. In \Cref{Section 2} we recall some useful concepts from graph theory, Stanley-Reisner theory, and related topics. In \Cref{Section 3} we give a proof of \Cref{intro main theorem}. Using this result, in \Cref{Section 4}, we provide explicit formulas for all the $\N$-graded Betti numbers of $R_G/I_r(G)$ for some families of graphs $G$.


\section{Preliminaries}\label{Section 2}

In this section, we recall some basic concepts from graph theory and commutative algebra.

A {\it simplicial complex} $\Delta$ on the vertex set $V(\Delta)$ is a collection of subsets of $V(\Delta)$ satisfying the following two conditions (i) $\{x_i\}\in \Delta$ for each $x_i\in V(\Delta)$; (ii) if $F\in \Delta$ and $F'\subseteq F$ then $F'\in \Delta$. The elements of $\Delta$ are called the {\it faces} of $\Delta$ and the maximal faces are called the {\it facets} of $\Delta$. If $F_1,\ldots, F_t$ are some faces of $\Delta$ such that for each $F\in\Delta$, $F\subseteq F_i$ for some $i$, then we say that $\Delta$ is generated by the faces $F_1,\ldots, F_t$ and write $\Delta=\langle F_1,\ldots, F_t \rangle$. For $F\in\Delta$, the {\it dimension} of $F$ is defined as $\dim F=|F|-1$ and we define $\dim\Delta=\max_{F\in \Delta}\dim F$. If $\Delta$ is a simplicial complex such that it contains exactly one maximal face then $\Delta$ is called a {\it simplex} on the vertex set $V(\Delta)$. If $\Delta$ contains exactly one maximal face and $\dim\Delta=d$ then it is called a {\it $d$-simplex}. For a simplicial complex $\Delta$ the {\it $k$-skeleton} of $\Delta$, denoted by $\Delta^k$, is the simplicial complex on the vertex set $V(\Delta)$ with faces $\{F\in\Delta\mid \dim F\le k\}$.

Let $\H$ be a hypergraph on the vertex set $\V$ and edge set $\E$. A subset $A\subseteq \V$ is called an independent set if $E\nsubseteq A$ for each $E\in \E$. The collection of all independent sets forms a simplicial complex called the {\it independence complex} of $\H$ and it is denoted by $\mathrm{Ind}(\H)$.

Let $R=\K[x_1,\ldots,x_n]$ be the polynomial ring on $n$ variables over a field $\K$. If $F\subseteq \{x_1,\ldots,x_n\}$ then by $\mathbf{x}_F$ we denote the monomial $\prod_{x_i\in F}x_i$ in $R$. A square-free monomial ideal $I$ of $R$ can be associated with a simplicial complex $\Delta(I)$ on the vertex set $\{x_1,\ldots,x_n\}$, whose faces are $\{F\mid \mathbf x_F\notin I\}$. The complex $\Delta(I)$ is called the {\it Stanley-Reisner complex} of $I$. Conversely, given a simplicial complex $\Delta$, the ideal $I_{\Delta}=\langle 
\mathbf{x}_{F}\mid F\notin \Delta \rangle$ is called the {\it Stanley-Reisner ideal} of $\Delta$. Thus we have a one-to-one correspondence between simplicial complexes on the vertex set $\{x_1,\ldots,x_n\}$ and square-free monomial ideals in $\K[x_1,\ldots,x_n]$ and this correspondence is called the {\it Stanley-Reisner correspondence}. If $\H$ is a hypergraph on the vertex set $\V=\{x_1,\ldots,x_n\}$ and edge set $\E$ then the Stanley Reisner ideal of $\mathrm{Ind}(H)$ is the ideal $I(\H)=\langle \mathbf{x}_E:=\prod_{x_i\in E}x_i\mid E\in \E \rangle$. The ideal $I(\H)$ is commonly called as the {\it (hyper)edge ideal} of $\H$ and this provides a one-to-one correspondence between the hypergraphs on the vertex set $\{x_1,\ldots,x_n\}$ and the square-free monomial ideals in $\K[x_1,\ldots,x_n]$.

In this paper, we are interested to study the Betti numbers and regularity of $R_G/I(\C)$, where $G$ is any finite simple graph and $r$ is a positive integer. Recall that for a homogeneous ideal $I \subseteq R = \mathbb{K}[x_1,\ldots,x_n]$, the {\it graded minimal free resolution} of $R/I$ is the long exact sequence 
$$0\rightarrow \bigoplus_{j \in \mathbb{N}} R(-j)^{\beta_{p,j}(R/I)}
\rightarrow \bigoplus_{j \in \mathbb{N}} R(-j)^{\beta_{p-1,j}(R/I)}
\rightarrow \cdots \rightarrow \bigoplus_{j \in \mathbb{N}}
R(-j)^{\beta_{1,j}(R/I)} \rightarrow R \rightarrow R/I \rightarrow 0,$$
where $R(-j)$ denotes the polynomial ring $\mathbb K[x_1,\ldots,x_n]$ with the grading
twisted by $j$, and $p \leq n$.  The numbers $\beta_{i,j}(R/I)$ are called the {\it $(i,j)$-th graded Betti numbers of $R/I$}. See \cite{RHV} for more on graded free resolution.

\begin{definition}
 The {\it Castelnuovo-Mumford regularity}, or simply the regularity of $R/I$ is defined
    to be ${\rm reg}(R/I) = \max\{j-i ~|~ \beta_{i,j}(R/I) \neq 0\}$.
\end{definition}

 One of the most important tools for computing the Betti numbers of a hyperedge ideal or equivalently any square-free monomial ideal associated to a simplicial complex is the following formula by Hochster.

\begin{proposition}\textup{\cite[Hochster's formula]{MH}}\label{Hochster}
Let $\Delta$ be a simplicial complex on the vertex set $\{x_1,\ldots,x_n\}$. Let $R/I_{\Delta}$ denote its Stanley-Reisner ring in the polynomial ring $R=\mathbb K[x_1,\ldots,x_n]$. Then for $i \geq 0$, the Betti numbers $\beta_{i,d}$ of $R/I_{\Delta}$ are given by 
{\footnotesize 
\[ \beta_{i,d}(R/I_{\Delta})= \sum_{\underset{|W|=d}{W \subseteq V(\Delta)}} \dim_{\mathbb{K}} \widetilde{H}_{d-i-1}(\Delta[W]; \mathbb{K}).\]} Here $\Delta[W]=\{F\in\Delta\mid F\subseteq W\}$  is the simplicial subcomplex of $\Delta$ induced on the vertex set $W$.
\end{proposition}

In our case $\Delta=\mathrm{Ind}(\C)=\mathrm{Ind}_r(G)$ and $I_{\Delta}=I(\C)=I_r(G)$, where $\mathrm{Ind}_r(G)$ is the $r$-independence complex of $G$. A subset $A\subseteq V(G)$ is called $r$-independent if each connected component of the induced subgraph $G[A]$ has vertex cardinality at most $r$. The collection of all $r$-independent sets forms the simplicial complex $\mathrm{Ind}_r(G)$. If $G=G_1*G_2$, then $\mathrm{Ind}_1(G)$ is the disjoint union of $\mathrm{Ind}_1(G_1)$ and $\mathrm{Ind}_1(G_2)$. However, for $r\ge 2$ the complex $\mathrm{Ind}_r(G)$ can be written as a union of two simplicial complexes with nonempty intersection (see \Cref{prop1}). By Hochster's formula, in order to calculate the Betti numbers $R_G/I_r(G)$ we need to calculate the homologies of the induced subcomplexes of $\mathrm{Ind}_r(G)$ and we do this by using the Mayer Vietoris sequences. 

Let $\Lambda$ be a simplicial complex with two simplicial subcomplexes $\Lambda_1$ and $\Lambda_2$ such that $\Lambda=\Lambda_1\cup\Lambda_2$. Then we have an exact sequence of (reduced) chain complexes 
\[
0\rightarrow \mathcal C_{\cdot}(\Lambda_1\cap\Lambda_2)\xrightarrow{\varphi_{\cdot}} \mathcal C_{\cdot}(\Lambda_1)\oplus \mathcal C_{\cdot}(\Lambda_2)\xrightarrow{\psi_{\cdot}} \mathcal C_{\cdot}(\Lambda)\rightarrow 0,
\]
where the maps $\varphi_{\cdot}$ and $\psi_{\cdot}$ are defined by $\varphi_{\cdot}(x)=(x,-x)$ and $\psi_{\cdot}(y,z)=y+z$, respectively. This induces the following long exact sequence of reduced homologies, called the Mayer-Vietoris sequence.
\[
\cdots\rightarrow \widetilde H_n(\Lambda_1\cap\Lambda_2)\rightarrow \widetilde H_n(\Lambda_1)\oplus\widetilde H_n(\Lambda_2)\rightarrow \widetilde H_n(\Lambda)\xrightarrow{\partial}\widetilde H_{n-1}(\Lambda_1\cap\Lambda_2)\rightarrow\cdots,
\]
where $\partial$ is the connecting homomorphism. For more on Mayer-Vietoris sequence one can consult \cite[Section 2.2]{AH}. In the next section, we make extensive use of the Mayer-Vietoris sequence to compute the Betti numbers of $I_r(G)$, where $G=G_1*G_2$.


\section{Betti numbers of the hyperedge ideal}\label{Section 3}
In this section we give a proof of \Cref{intro main theorem}, i.e.,  we express the graded Betti numbers $\beta_{i,d}(R_{G_1*G_2}/I_r(G_1*G_2))$ in terms of the graded Betti numbers of $R_{G_1}/I_r(G_1)$ and $R_{G_2}/I_r(G_2)$. The first step to achieve this is to determine the simplicial complex $\rindj$. 

\begin{proposition}\label{prop1}
    Let $G_1$ and $G_2$ be two graphs on the vertex sets $\{x_1,\ldots,x_n\}$ and $\{y_1,\ldots,y_m\}$, respectively. Moreover, let $\Gamma$ denote the simplicial complexes $\mathrm{Ind}_r(G_1*G_2)$. Then 
    {\footnotesize
    \begin{align*}
        \Gamma&=\mathrm{Ind}_r(G_1)\cup \mathrm{Ind}_r(G_2)\cup \left\langle\left\{S\sqcup T\mid S\subseteq V(G_1), T\subseteq V(G_2), 1\le |S|, 1\le |T|, \text{ and } |S|+|T|\le r\right\}\right\rangle .
    \end{align*}
}
\end{proposition}

\begin{proof}
    Let $F\in \Gamma$. If $F\subseteq V(G_i)$ for some $i$, then $F$ must be an $r$-independent set of $G_i$. Hence $F\in \mathrm{Ind}_r(G_i)$. Now let $F\cap V(G_i)\neq\emptyset$ for each $i\in [2]$. Then for each $x_u\in F\cap V(G_1)$ and $x_v\in F\cap V(G_2)$, $\{x_u,x_v\}\in E(G_1*G_2)$. Therefore, $(G_1*G_2)[F]$ is connected and hence $|F|\le r$. Thus if $F$ is a facet of $\Gamma$ with $F\cap V(G_i)\neq\emptyset$ for each $i\in [2]$, then $|F|=r$ and $F=S\sqcup T$, where $S=F\cap V(G_1)$ and $T=F\cap V(G_2)$ with $|S|,|T|\ge 1$ and $|S|+|T|=r$. This completes the proof.
\end{proof}

By Hochster's formula, we have
{\footnotesize
\[
\beta_{i,i+d}(R_{G_1*G_2}/I_r(G_1*G_2))= \sum_{\underset{|W|=i+d}{W \subseteq\{x_1,\ldots,x_n,y_1,\ldots,y_m\} }} \dim_{\mathbb{K}} \widetilde{H}_{d-1}(\Gamma[W]; \mathbb{K}),
\]}
where $\Gamma[W]=\{F\in \Gamma\mid F\subseteq W\}$ is the induced subcomplex of $\Gamma$ on the vertex set $W\subseteq V(G_1*G_1)=\{x_1,\ldots,x_n,y_1,\ldots,y_m\}$. Note that as $I_r(G_1*G_2)$ is a monomial ideal equigenerated in degree $r+1$, the only possible nonzero Betti numbers are $\beta_{i,i+d}(R_{G_1*G_2}/I_r(G_1*G_2))$, where $d\ge r$ (cf. \cite[Proposition 12.3]{IP} ). Thus it is enough to find $\dim_{\mathbb K}\widetilde H_{d-1}(\Gamma[W];\mathbb K)$ for $d\ge r$. In order to determine the homologies of the induced subcomplexes of $\Gamma$, we first divide $\Gamma$ into smaller subcomplexes $\Gamma_1$ and $\Gamma_2$ in the following way.
{\footnotesize
\begin{align*}
    \Gamma_1&=\mathrm{Ind}_r(G_1)\cup \left\langle\left\{S\sqcup T\mid S\subseteq V(G_1), T\subseteq V(G_2), 1\le |S|, 1\le |T|,\text{ and } |S|+|T|\le r\right\}\right\rangle,\\
        \Gamma_2&=\mathrm{Ind}_r(G_2)\cup \left\langle\left\{S\sqcup T\mid S\subseteq V(G_1), T\subseteq V(G_2), 1\le |S|, 1\le |T|,\text{ and } |S|+|T|\le r\right\}\right\rangle.
\end{align*}}
It is easy to see that $\Gamma=\Gamma_1\cup\Gamma_2$. Moreover,
{\footnotesize
\begin{align*}
    \Gamma_1\cap\Gamma_2= \left\langle\left\{S\sqcup T\mid S\subseteq V(G_1), T\subseteq V(G_2),1\le |S|, 1\le |T|, \text{ and } |S|+|T|\le r\right\}\right\rangle.
\end{align*}
}
We further subdivide $\Gamma_1$ and $\Gamma_2$ in the following way. Let $\Gamma_1=\Gamma_{11}\cup \Gamma_{12}$ and $\Gamma_2=\Gamma_{21}\cup\Gamma_{22}$, where $\Gamma_{11}=\mathrm{Ind}_r(G_1)$, $\Gamma_{22}=\mathrm{Ind}_r(G_2)$ and $\Gamma_{12}=\Gamma_{21}=\Gamma_1\cap\Gamma_{2}$. Note that, $\Gamma_{11}\cap\Gamma_{12}=\langle \{S\subseteq V(G_1)\mid |S|=r-1\} \rangle$. Similarly, $\Gamma_{21}\cap\Gamma_{22}=\langle \{T\subseteq V(G_2)\mid |T|=r-1\} \rangle$.

The following proposition describes how the induced subcomplex $\Gamma[W]$ for $W\subseteq V(\Gamma)$, gets divided into the corresponding induced subcomplexes.

\begin{proposition}\label{subcomplex on subset}
    Let $\Delta$ be a simplicial complex on the vertex set $V$. Let $\Delta_1$ and $\Delta_2$ be two simplicial subcomplexes of $\Delta$ on the vertex sets $V_1$ and $V_2$, respectively, such that $\Delta=\Delta_1\cup\Delta_2$. Then for any $W\subseteq V$,
\begin{enumerate}[(i)]
    \item $\Delta[W]=\Delta_1[W\cap V_1]\cup\Delta_2[W\cap V_2]$;
    \item $\Delta_1[W]\cap\Delta_2[W]=(\Delta_1\cap\Delta_2)[W]$;
    \item $\Delta_1[W]=\Delta_1[W\cap V_1]$ and $\Delta_2[W]=\Delta_2[W\cap V_2]$.
\end{enumerate}
\end{proposition}

\begin{proof}
    Let $\tau\in \Delta[W]$. Then $\tau\subseteq W\subseteq V$ and $\tau\in\Delta$. Since $\Delta=\Delta_1\cup\Delta_2$, either $\tau\in\Delta_1$ or $\tau\in\Delta_2$. Therefore, either $\tau\subseteq W\cap V_1$ or $\tau\subseteq W\cap V_2$. Thus we have $(i)$. For $(ii)$, it is enough to observe that if $\tau\in \Delta_1[W]\cap\Delta_2[W]$, then $\tau\in\Delta_1\cap\Delta_2$ and $\tau\subseteq W$. Also, $(iii)$ follows from the definition of induced subcomplexes.
\end{proof}

We now proceed to compute the $\K$-vector space dimension of reduced homologies of various induced subcomplexes.

\begin{lemma}\label{intersection lemma}
    Let $W_1\subseteq V(G_1)$ and $W_2\subseteq V(G_2)$ such that $|W_1|=t_1\ge 1$ and $|W_2|=t_2\ge 1$, respectively. If $W=W_1\sqcup W_2\subseteq V(G_1*G_2)$, then
    {\footnotesize
    \[
    \mathrm{dim}_{\K}(\widetilde H_d((\Gamma_1\cap\Gamma_2)[W];\mathbb K)=\begin{cases}
        {t_1+t_2-1\choose r}-{t_1\choose r}-{t_2\choose r}&\text{ if }d=r-1\\
        0 &\text{ otherwise}.
    \end{cases}
    \]}
\end{lemma}

\begin{proof}
    We first show that $\intersectionw$ is a Cohen-Macaulay simplicial complex. Note that
    {\footnotesize
\begin{align*}
    (\Gamma_1\cap\Gamma_2)[W]= \left\langle\left\{S\sqcup T\mid S\subseteq W_1, T\subseteq W_2, 1\le |S|, 1\le |T|, \text{ and } |S|+|T|\le r\right\}\right\rangle.
\end{align*}}

Let $K_{W_1}$ and $K_{W_2}$ denote the complete graph on the vertex sets $W_1$ and $W_2$, respectively. Then $\intersectionw=\Delta^{r-1}$, where $\Delta=\mathrm{Ind}_{r-1}(K_{W_1}\sqcup K_{W_2})$. It is easy to see that if $t_1+t_2\le r$, then $\intersectionw$ is a simplex. Now let $I$, $I_1$ and $I_2$ denote the Stanley-Reisner ideals of $\Delta$, $\mathrm{Ind}_{r-1}(K_{W_1})$ and $\mathrm{Ind}_{r-1}(K_{W_2})$, respectively. Then $I=I_1+I_2$. Note that $\mathrm{Ind}_{r-1}(K_{W_1})$ is a Cohen-Macaulay simplicial complex since it is the $(r-2)$-skeleton of the simplex on the vertex set $|W_1|$. Similarly, $\mathrm{Ind}_{r-1}(K_{W_2})$ is also a Cohen-Macaulay simplicial complex. Let $R_{W_1}=\K[x_i\mid x_i\in W_1]$ and $R_{W_2}=\K[x_i\mid x_i\in W_2]$. Then $R_{W_1}/I_1$ is a Cohen-Macaulay $R_{W_1}$-module. Similarly, $R_{W_2}/I_2$ is a Cohen-Macaulay $R_{W_2}$-module. We have $R_{W_1}/I_1\otimes_{\K}R_{W_2}/I_2\cong R_W/I$, where $R_W=\K[x_i\mid x_i\in W]$ (by \cite[Proposition 2.2.20]{RHV}). Therefore, by \cite[Corollary 2.2.22]{RHV}, $R_W/I$ is a Cohen-Macaulay $R$-module. Since $I$ is the Stanley-Reisner ideal of $\Delta$, the simplicial complex $\Delta$ is Cohen-Macaulay. Hence by \cite[Proposition 5.3.14]{RHV}, $\intersectionw=\Delta^{r-1}$ is a Cohen-Macaulay simplicial complex.

If $t_1+t_2\le r$, then $\intersectionw$ is a simplex and hence all the homologies are zero. Therefore, we may assume that $t_1+t_2>r$. In that case $\dim(\intersectionw)=r-1$. Let $a_j$ denote the number of $j$-dimensional simplices of $\intersectionw$. Then $a_{r-1}={t_1+t_2\choose r}-{t_1\choose r}-{t_2\choose r}$. Also, for $0\le j\le r-2$, $a_j={t_1+t_2\choose j+1}$. Therefore, the Euler characteristic of the simplicial complex $\intersectionw$ is
{\footnotesize
\[
\sum_{j=0}^{r-1}(-1)^j{t_1+t_2\choose j+1}+(-1)^r\left[{t_1\choose r}+{t_2\choose r}\right]=1+(-1)^{r-1}{t_1+t_2-1\choose r}+(-1)^r\left[{t_1\choose r}+{t_2\choose r}\right].
\]
}
Since the homologies of $\intersectionw$ are concentrated in dimensions $0$ and $r-1$, the Euler characteristic of $\intersectionw$ is the number $\dim_{\K}(H_0((\Gamma_1\cap\Gamma_2)[W];\K))+(-1)^{r-1}\dim_{\K}(H_{r-1}((\Gamma_1\cap\Gamma_2)[W];\K))$. Therefore, $\dim_{\K}(\widetilde H_{r-1}((\Gamma_1\cap\Gamma_2)[W];\K))={t_1+t_2-1\choose r}-{t_1\choose r}-{t_2\choose r}$. This completes the proof.
\end{proof}

\begin{lemma}\label{further intersection lemma1}
    Let $W_1\subseteq V(G_1)$ such that $|W_1|=t_1\ge 1$. Then
   {\footnotesize
    \[
    \mathrm{dim}_{\K}(\widetilde H_d((\Gamma_{11}\cap\Gamma_{12})[W_1];\mathbb K)=\begin{cases}
        {t_1-1\choose r-1}&\text{ if }d=r-2\\
        0 &\text{ otherwise}.
    \end{cases}
    \]}
\end{lemma}

\begin{proof}
    Note that $(\Gamma_{11}\cap\Gamma_{12})[W_1]=\langle \{S\subseteq W_1\mid |S|\le r-1\} \rangle$. Thus if $t_1<r$, then $(\Gamma_{11}\cap\Gamma_{12})[W_1]$ is a simplex. Therefore, we may assume that $t_1\ge r$. In that case, $(\Gamma_{11}\cap\Gamma_{12})[W_1]$ is the $(r-2)$-skeleton of the simplex on the vertex set $W_1$. Therefore, $(\Gamma_{11}\cap\Gamma_{12})[W_1]$ is a Cohen-Macaulay simplicial complex. Thus the homologies of $(\Gamma_{11}\cap\Gamma_{12})[W_1]$ are concentrated in dimensions $0$ and $r-2$. Let $a_j$ denote the number of $j$-dimensional simplices of $(\Gamma_{11}\cap\Gamma_{12})[W_1]$. Then for $0\le j\le r-2$, $a_j={t_1\choose j+1}$. Proceeding in the similar way as in Lemma \ref{intersection lemma}, we have $\mathrm{dim}_{\K}(\widetilde H_{r-2}((\Gamma_{11}\cap\Gamma_{12})[W_1];\mathbb K)={t_1-1\choose r-1}$.
\end{proof}

Similarly, we have the following lemma for induced subcomplexes of  $\Gamma_{21}\cap\Gamma_{22}$.
\begin{lemma}\label{further intersection lemma2}
    Let $W_2\subseteq V(G_2)$ such that $|W_2|=t_2\ge 1$. Then
   {\footnotesize
    \[
    \mathrm{dim}_{\K}(\widetilde H_d((\Gamma_{21}\cap\Gamma_{22})[W_2];\mathbb K)=\begin{cases}
        {t_2-1\choose r-1}&\text{ if }d=r-2\\
        0 &\text{ otherwise}.
    \end{cases}
    \]}
\end{lemma}
Now we proceed to determine the homologies of the induced subcomplexes of $\Gamma$.

\begin{lemma}\label{r-1 homology lemma}
    Let $W_1\subseteq V(G_1)$ and $W_2\subseteq V(G_2)$ such that $|W_1|=t_1\ge 1$ and $|W_2|=t_2\ge 1$, respectively. If $W=W_1\sqcup W_2\subseteq V(G_1*G_2)$, then
{\footnotesize
    \begin{align*}
     &\dim_{\K}(\widetilde H_{r-1}(\Gamma[W];\K))\\
     &=\dim_{\K}(\widetilde H_{r-1}(\Gamma_{11}[W_1];\K))+\dim_{\K}(\widetilde H_{r-1}(\Gamma_{22}[W_2];\K))+{t_1+t_2-1\choose r}-{t_1-1\choose r}-{t_2-1\choose r}.   
    \end{align*}}
     \end{lemma}

\begin{proof}
    Using \Cref{subcomplex on subset} for the simplicial complex $\Gamma=\Gamma_1\cup\Gamma_2$, we have $\Gamma[W]=\Gamma_1[W]\cup\Gamma_2[W]$. Moreover, $\Gamma_1[W]\cap\Gamma_2[W]=(\Gamma_1\cap\Gamma_2)[W]$. Similarly, $\Gamma_1[W]=\Gamma_{11}[W_1]\cup\Gamma_{12}[W]$, where $\Gamma_{11}[W_1]\cap\Gamma_{12}[W]=(\Gamma_{11}\cap\Gamma_{12})[W_1]$. Also, $\Gamma_2[W]=\Gamma_{21}[W]\cup\Gamma_{22}[W_2]$, where $\Gamma_{21}[W]\cap\Gamma_{22}[W_2]=(\Gamma_{21}\cap\Gamma_{22})[W_2]$.

    Using the Mayer-Vietoris sequence for $\Gamma[W]=\Gamma_1[W]\cup\Gamma_2[W]$, we have the following long exact sequence.
{\footnotesize
    \begin{equation}\label{eq1}
    \begin{split}
        \cdots \rightarrow \widetilde H_r(\Gamma[W];\K)\xrightarrow{\partial^r} \widetilde H_{r-1}((\Gamma_1\cap\Gamma_2)[W];\K)\rightarrow \widetilde H_{r-1}(\Gamma_1[W];\K)\oplus \widetilde H_{r-1}(\Gamma_2[W];\K) \\
        \rightarrow \widetilde H_{r-1}(\Gamma[W];\K)\xrightarrow{\partial^{r-1}} \widetilde H_{r-2}((\Gamma_1\cap\Gamma_2)[W];\K)\rightarrow\cdots.
    \end{split}
    \end{equation}}
    We proceed to show that the map $\partial^r$ in \Cref{eq1} is the zero map. We have the following chain of maps.
    
    {\scriptsize
\begin{tikzcd}[cells={nodes={minimum height=2em}}]
& \vdots\arrow[d]&\vdots\arrow[d] &\vdots\arrow[d] &\\
0\arrow[r] & C_r((\Gamma_1\cap\Gamma_2)[W];\K) \arrow[r,"\varphi_r"] \arrow[d,"\delta_r^{(\Gamma_1\cap\Gamma_2)[W]}"] &  C_r(\Gamma_1[W];\K)\oplus C_r(\Gamma_2[W];\K)) \arrow[r,"\psi_r"] \arrow[d,"\theta_r^{(\Gamma_1[W],\Gamma_2[W])}"] &  C_r(\Gamma[W];\K) \arrow[r]\arrow[d,"\delta_r^{\Gamma[W]}"] & 0\\
0 \arrow[r] & C_{r-1}((\Gamma_1\cap\Gamma_2)[W];\K) \arrow[r,"\varphi_{r-1}"] \arrow[d,"\delta_{r-1}^{(\Gamma_1\cap\Gamma_2)[W]}"]  & C_{r-1}(\Gamma_1[W];\K)\oplus C_{r-1}(\Gamma_2[W];\K)) \arrow[r,"\psi_{r-1}"] \arrow[d,"\theta_{r-1}^{(\Gamma_1[W],\Gamma_2[W])}"] & C_{r-1}(\Gamma[W];\K) \arrow[r]\arrow[d,"\delta_{r-1}^{\Gamma[W]}"] & 0\\
 0 \arrow[r] & C_{r-2}((\Gamma_1\cap\Gamma_2)[W];\K) \arrow[r,"\varphi_{r-2}"] \arrow[d]  & C_{r-2}(\Gamma_1[W];\K)\oplus C_{r-2}(\Gamma_2[W];\K)) \arrow[r,"\psi_{r-2}"] \arrow[d] & C_{r-2}(\Gamma[W];\K) \arrow[r]\arrow[d] & 0\\
 & \vdots&\vdots &\vdots &
\end{tikzcd}}

Here $\delta_i$'s are the corresponding chain maps. Also, $\varphi_i(x)=(x,-x)$ for $x\in C_i((\Gamma_1\cap\Gamma_2)[W];\K)$, $\psi_i(y,z)=y+z$, and $\theta_i^{(\Gamma_1[W],\Gamma_2[W])}(y,z)=(\delta_i^{\Gamma_1[W]}(y),\delta_i^{\Gamma_2[W]}(z))$ for $y\in C_i(\Gamma_1[W];\K), z\in C_i(\Gamma_2[W];\K)$. Now let $[u]\in \widetilde H_r(\Gamma[W];\K)$. Then $u\in \ker \delta_r^{\Gamma[W]}$. Hence $u=\psi_r((a,b))=a+b$, for some $a\in C_r(\Gamma_1[W];\K) $ and $b\in C_r(\Gamma_2[W];\K)$. Therefore, $a=\sum c_{i_1,\ldots,i_{r+1}}e_{\{x_{i_1},\ldots,x_{i_{r+1}}\}}$ and $b=\sum d_{j_1,\ldots,j_{r+1}}e_{\{y_{j_1},\ldots,y_{j_{r+1}}\}}$, where $e_{\{x_{i_1},\ldots,x_{i_{r+1}}\}}\in C_r(\Gamma_1[W];\K)$, $e_{\{y_{j_1},\ldots,y_{j_{r+1}}\}}\in C_r(\Gamma_2[W];\K)$, and $c_{i_1,\ldots,i_{r+1}},d_{j_1,\ldots,j_{r+1}}\in\K$. Note that $\delta_r^{\Gamma[W]}(\psi_r((a,b)))=0$. In other words, 
{\footnotesize
\[
\sum c_{i_1,\ldots,i_{r+1}}\delta_r^{\Gamma_1[W]}(e_{\{x_{i_1},\ldots,x_{i_{r+1}}\}})+\sum d_{j_1,\ldots,j_{r+1}}\delta_r^{\Gamma_2[W]}(e_{\{y_{j_1},\ldots,y_{j_{r+1}}\}})=0.
\]}
Now since $C_r(\Gamma[W])$ is a $\K$-vector space with a basis $\{e_{\{x_{i_1},\ldots,x_{i_{r+1}}\}},e_{\{y_{j_1},\ldots,y_{j_{r+1}}\}}\}$, where $\{x_{i_1},\ldots,x_{i_{r+1}}\}\subseteq W_1$ and $\{y_{j_1},\ldots,y_{j_{r+1}}\}\subseteq W_2$, we have $\delta_r^{\Gamma_1[W]}(a)=\delta_r^{\Gamma_2[W]}(b)=0$. Consequently, $\partial^r([u])=[\delta_r^{\Gamma_1[W]}(a)]=[\delta_r^{\Gamma_2[W]}(b)]=0$. Thus the map $\partial^r$ in \Cref{eq1} is the zero map. Also, by \Cref{intersection lemma}, $\widetilde H_{r-2}((\Gamma_1\cap\Gamma_2)[W];\K)=0$. Therefore, from \Cref{eq1} we have
{\footnotesize
\begin{align}\label{eq2}
    \widetilde H_{r-1}(\Gamma_1[W];\K)\oplus \widetilde H_{r-1}(\Gamma_2[W];\K)\cong \widetilde H_{r-1}((\Gamma_1\cap\Gamma_2)[W];\K)\oplus \widetilde H_{r-1}(\Gamma[W];\K).
\end{align}}

Now we proceed to obtain the $\K$-vector space dimensions of the homologies $\widetilde H_{r-1}(\Gamma_1[W];\K)$ and $\widetilde H_{r-1}(\Gamma_2[W];\K)$. Using the Mayer-Vietoris sequence for the simplicial complex $\Gamma_1[W]=\Gamma_{11}[W_1]\cup\Gamma_{12}[W]$ we have the following long exact sequence.
{\footnotesize
\begin{equation}\label{eq3}
    \begin{split}
        \cdots \rightarrow \widetilde H_{r-1}((\Gamma_{11}\cap\Gamma_{12})[W_1];\K)\rightarrow \widetilde H_{r-1}(\Gamma_{11}[W_1];\K)\oplus \widetilde H_{r-1}(\Gamma_{12}[W];\K)
        \rightarrow \widetilde H_{r-1}(\Gamma_1[W];\K)\\
        \xrightarrow{\partial^{r-1}} \widetilde H_{r-2}((\Gamma_{11}\cap\Gamma_{12})[W_1];\K)\rightarrow\cdots.
    \end{split}
    \end{equation}}

    In the next step, we show that the map $\partial^{r-1}$ in \Cref{eq3} is surjective. For this it is enough to show that $\widetilde H_{r-2}((\Gamma_{11}\cap\Gamma_{12})[W_1];\K)\xrightarrow{i^*} \widetilde H_{r-2}(\Gamma_{11}[W_1];\K)\oplus \widetilde H_{r-2}(\Gamma_{12}[W];\K)$ is the zero map. Again we have the following chain of maps.

     {\scriptsize
\begin{tikzcd}[cells={nodes={minimum height=2em}}]
& \vdots\arrow[d]&\vdots\arrow[d] &\vdots\arrow[d] &\\
0\arrow[r] & C_{r-1}((\Gamma_{11}\cap\Gamma_{12})[W_1];\K) \arrow[r,"\varphi_{r-1}"] \arrow[d,"\delta_{r-1}^{(\Gamma_{11}\cap\Gamma_{12})[W_1]}"] &  C_{r-1}(\Gamma_{11}[W_1];\K)\oplus C_{r-1}(\Gamma_{12}[W];\K)) \arrow[r,"\psi_{r-1}"] \arrow[d,"\theta_{r-1}^{(\Gamma_{11}[W_1],\Gamma_{12}[W])}"] &  C_{r-1}(\Gamma_1[W];\K) \arrow[r]\arrow[d,"\delta_{r-1}^{\Gamma_1[W]}"] & 0\\
0\arrow[r] & C_{r-2}((\Gamma_{11}\cap\Gamma_{12})[W_1];\K) \arrow[r,"\varphi_{r-2}"] \arrow[d,"\delta_{r-2}^{(\Gamma_{11}\cap\Gamma_{12})[W_1]}"] &  C_{r-2}(\Gamma_{11}[W_1];\K)\oplus C_{r-2}(\Gamma_{12}[W];\K)) \arrow[r,"\psi_{r-2}"] \arrow[d,"\theta_{r-2}^{(\Gamma_{11}[W_1],\Gamma_{12}[W])}"] &  C_{r-2}(\Gamma_1[W];\K) \arrow[r]\arrow[d,"\delta_{r-2}^{\Gamma_1[W]}"] & 0\\
0\arrow[r] & C_{r-3}((\Gamma_{11}\cap\Gamma_{12})[W_1];\K) \arrow[r,"\varphi_{r-3}"] \arrow[d] &  C_{r-3}(\Gamma_{11}[W_1];\K)\oplus C_{r-3}(\Gamma_{12}[W];\K)) \arrow[r,"\psi_{r-3}"] \arrow[d] &  C_{r-3}(\Gamma_1[W];\K) \arrow[r]\arrow[d] & 0\\
 & \vdots&\vdots &\vdots &
\end{tikzcd}}

Let $[u]\in \widetilde H_{r-2}((\Gamma_{11}\cap\Gamma_{12})[W_1];\K)$. Then $i^*([u])=([u],-[u])$, where $u$ is considered to be element of both $C_{r-2}(\Gamma_{11}[W_1];\K)$ and $C_{r-2}(\Gamma_{12}[W];\K))$. Let $\Delta_{W_1}$ denote the $(|W_1|-1)$-simplex on the vertex set $W_1$. Then $C_i(\Gamma_{11}[W_1])=C_i(\Delta_{W_1})$ for all $i\le r-1$. Thus $\widetilde H_{r-2}(\Gamma_{11}[W_1];\K)=\widetilde H_{r-2}(\Delta_{W_1};\K)=0$. Moreover, by \Cref{intersection lemma}, $\widetilde H_{r-2}(\Gamma_{12}[W];\K)=\widetilde H_{r-2}((\Gamma_{1}\cap\Gamma_2)[W];\K)=0$. Therefore, we have that $i^*$ is the zero map, and hence $\partial^{r-1}$ in \Cref{eq3} is surjective. By  \Cref{further intersection lemma1}, $\widetilde H_{r-1}((\Gamma_{11}\cap\Gamma_{12})[W_1];\K)=0$. Hence from \Cref{eq3} we have,
{\footnotesize
    \begin{align}\label{eq4}
        \widetilde H_{r-1}(\Gamma_1[W];\K)\cong \widetilde H_{r-1}(\Gamma_{11}[W_1];\K)\oplus \widetilde H_{r-1}(\Gamma_{12}[W];\K)\oplus \widetilde H_{r-2}((\Gamma_{11}\cap\Gamma_{12})[W_1];\K).
    \end{align}}
    A similar calculation for $\Gamma_2[W]=\Gamma_{22}[W_2]\cup\Gamma_{21}[W]$ yields the following.
{\footnotesize
\begin{align}\label{eq5}
        \widetilde H_{r-1}(\Gamma_2[W];\K)\cong \widetilde H_{r-1}(\Gamma_{22}[W_2];\K)\oplus \widetilde H_{r-1}(\Gamma_{21}[W];\K)\oplus \widetilde H_{r-2}((\Gamma_{21}\cap\Gamma_{22})[W_2];\K).
    \end{align}}

    Therefore, using \Cref{intersection lemma}, \Cref{further intersection lemma1}, \Cref{further intersection lemma2}, \Cref{eq2}, \Cref{eq4} and \Cref{eq5}, we have,

{\footnotesize
    \begin{align*}
     &\dim_{\K}(\widetilde H_{r-1}(\Gamma[W];\K))\\
     &=\dim_{\K}(\widetilde H_{r-1}(\Gamma_{11}[W_1];\K))+\dim_{\K}(\widetilde H_{r-1}(\Gamma_{22}[W_2];\K))+{t_1+t_2-1\choose r}-{t_1-1\choose r}-{t_2-1\choose r}.   
    \end{align*}}
\end{proof}

\begin{lemma}\label{i homology lemma}
    Let $W_1\subseteq V(G_1)$ and $W_2\subseteq V(G_2)$ such that $|W_1|=t_1\ge 1$ and $|W_2|=t_2\ge 1$, respectively. If $W=W_1\sqcup W_2\subseteq V(G_1*G_2)$, then for all $i\ge r$,
{\footnotesize
    \begin{align*}
     \dim_{\K}(\widetilde H_{i}(\Gamma[W];\K))=\dim_{\K}(\widetilde H_{i}(\Gamma_{11}[W_1];\K))+\dim_{\K}(\widetilde H_{i}(\Gamma_{22}[W_2];\K)).   
    \end{align*}}
     \end{lemma}

\begin{proof}
The proof follows the same line of arguments as in \Cref{r-1 homology lemma}.

Using the Mayer-Vietoris sequence for $\Gamma[W]=\Gamma_1[W]\cup\Gamma_2[W]$, we have the following long exact sequence of $\mathbb K$-vector spaces.
{\footnotesize
    \begin{equation}\label{eq6}
    \begin{split}
        \cdots \rightarrow \widetilde H_{r+1}(\Gamma[W];\K)\xrightarrow{\partial^{r+1}} \widetilde H_{r}((\Gamma_1\cap\Gamma_2)[W];\K)\rightarrow \widetilde H_{r}(\Gamma_1[W];\K)\oplus \widetilde H_{r}(\Gamma_2[W];\K) \\
        \rightarrow \widetilde H_{r}(\Gamma[W];\K)\xrightarrow{\partial^{r}} \widetilde H_{r-1}((\Gamma_1\cap\Gamma_2)[W];\K)\rightarrow\cdots.
    \end{split}
    \end{equation}}      
Note that $\widetilde H_{r}((\Gamma_1\cap\Gamma_2)[W];\K)=0$ (by \Cref{intersection lemma}). Moreover, by the proof of \Cref{r-1 homology lemma}, we see that $\partial^r$ in \Cref{eq6} is the zero map. Hence
{\footnotesize
\begin{align}\label{eq7}
    \widetilde H_{r}(\Gamma[W];\K)\cong \widetilde H_{r}(\Gamma_1[W];\K)\oplus \widetilde H_{r}(\Gamma_2[W];\K).
\end{align}}
For $\Gamma_1[W]=\Gamma_{11}[W_1]\cup\Gamma_{12}[W]$ we have the following Mayer-Vietoris sequence
{\footnotesize
\begin{equation}\label{eq8}
    \begin{split}
        \cdots \rightarrow \widetilde H_{r}((\Gamma_{11}\cap\Gamma_{12})[W_1];\K)\rightarrow \widetilde H_{r}(\Gamma_{11}[W_1];\K)\oplus \widetilde H_{r}(\Gamma_{12}[W];\K)
        \rightarrow \widetilde H_{r}(\Gamma_1[W];\K)\\
        \xrightarrow{\partial^{r}} \widetilde H_{r-1}((\Gamma_{11}\cap\Gamma_{12})[W_1];\K)\rightarrow\cdots.
    \end{split}
    \end{equation}}
    By \Cref{further intersection lemma1}, $\widetilde H_{r}((\Gamma_{11}\cap\Gamma_{12})[W_1];\K)=\widetilde H_{r-1}((\Gamma_{11}\cap\Gamma_{12})[W_1];\K)=0$. Also, by \Cref{intersection lemma}, we have $\widetilde H_{r}(\Gamma_{12}[W];\K)=\widetilde H_{r}((\Gamma_{1}\cap\Gamma_2)[W];\K)=0$. Therefore, from \Cref{eq8} we obtain $\widetilde H_{r}(\Gamma_1[W];\K)\cong \widetilde H_{r}(\Gamma_{11}[W_1];\K)$. A similar calculation for $\Gamma_2[W]=\Gamma_{21}[W]\cup\Gamma_{22}[W_2]$ yields $\widetilde H_{r}(\Gamma_2[W];\K)\cong \widetilde H_{r}(\Gamma_{22}[W_2];\K)$. Thus from \Cref{eq7} we get that $\widetilde H_{r}(\Gamma[W];\K)\cong \widetilde H_{r}(\Gamma_{11}[W_1];\K)\oplus \widetilde H_{r}(\Gamma_{22}[W_2];\K)$.

    Now let $i\ge r+1$. For the Mayer-Vietoris sequence
    {\footnotesize
    \begin{equation*}
    \begin{split}
        \cdots \rightarrow\widetilde H_{i}((\Gamma_1\cap\Gamma_2)[W];\K)\rightarrow \widetilde H_{i}(\Gamma_1[W];\K)\oplus \widetilde H_{i}(\Gamma_2[W];\K)
        \rightarrow \widetilde H_{i}(\Gamma[W];\K)\\
        \xrightarrow{\partial^{i}} \widetilde H_{i-1}((\Gamma_1\cap\Gamma_2)[W];\K)\rightarrow\cdots,
    \end{split}
    \end{equation*}} 
    we have $\widetilde H_{i}((\Gamma_1\cap\Gamma_2)[W];\K)=\widetilde H_{i-1}((\Gamma_1\cap\Gamma_2)[W];\K)=0$ (by \Cref{intersection lemma}). Therefore, $\widetilde H_{i}(\Gamma[W];\K)\cong \widetilde H_{i}(\Gamma_1[W];\K)\oplus \widetilde H_{i}(\Gamma_2[W];\K)$. Now for the Mayer-Vietoris sequence 
{\footnotesize
\begin{equation*}
    \begin{split}
        \cdots \rightarrow \widetilde H_{i}((\Gamma_{11}\cap\Gamma_{12})[W_1];\K)\rightarrow \widetilde H_{i}(\Gamma_{11}[W_1];\K)\oplus \widetilde H_{i}(\Gamma_{12}[W];\K)
        \rightarrow \widetilde H_{i}(\Gamma_1[W];\K)\\
        \xrightarrow{\partial^{i}} \widetilde H_{i-1}((\Gamma_{11}\cap\Gamma_{12})[W_1];\K)\rightarrow\cdots,
    \end{split}
    \end{equation*}}
we have $\widetilde H_{i}((\Gamma_{11}\cap\Gamma_{12})[W_1];\K)=\widetilde H_{i-1}((\Gamma_{11}\cap\Gamma_{12})[W_1];\K)=0$ (by \Cref{further intersection lemma1}). Moreover, $\widetilde H_{i}(\Gamma_{12}[W];\K)=\widetilde H_{i}((\Gamma_{1}\cap\Gamma_2)[W];\K)=0$ (by \Cref{intersection lemma}).  Therefore, from the long exact sequence above we get that $\widetilde H_{i}(\Gamma_1[W];\K)\cong \widetilde H_{i}(\Gamma_{11}[W_1];\K)$. Proceeding as above for the simplicial complex $\Gamma_2[W]$, we have $\widetilde H_{i}(\Gamma_2[W];\K)\cong \widetilde H_{i}(\Gamma_{22}[W_2];\K)$. Consequently, we obtain $\widetilde H_{i}(\Gamma[W];\K)\cong \widetilde H_{i}(\Gamma_{11}[W_1];\K)\oplus \widetilde H_{i}(\Gamma_{22}[W_2];\K)$ for all $i\ge r+1$. This completes the proof.

\end{proof}

Now we are ready to prove the main result of this section.

\begin{theorem}\label{main theorem}
    Let $G_1$ and $G_2$ be two graphs on the vertex sets $\{x_1,\ldots,x_n\}$ and $\{y_1,\ldots,y_m\}$, respectively. Let $r$ be a positive integer. Then the $\mathbb N$-graded Betti numbers of $R_{G_1*G_2}/I_r(G_1*G_2)$ can be expressed as follows. {\footnotesize
    \begin{align*}
    &\beta_{i,i+d}(R_{G_1*G_2}/I_r(G_1*G_2))\\
    &=\begin{cases}
        \sum_{j=0}^{i+d-1}\left\{{m\choose j}\beta_{i-j,i-j+d}(R_{G_1}/I_r(G_1))+{n\choose j}\beta_{i-j,i-j+d}(R_{G_2}/I_r(G_2)) \right\}+\\
\hspace{11em}   \sum_{j=1}^{i+d-1} \left[{i+d-1\choose d}-{i+d-1-j\choose d}-{j-1\choose d}\right]{m\choose j}{n\choose i+d-j}\hspace{2em}\text{ for }d=r,
 \\
    \sum_{j=0}^{i+d-1}\left\{{m\choose j}\beta_{i-j,i-j+d}(R_{G_1}/I_r(G_1))+{n\choose j}\beta_{i-j,i-j+d}(R_{G_2}/I_r(G_2)) \right\}\hspace{4.7em}\text{ for }d\ge r+1.
    \end{cases}
    \end{align*}}
\end{theorem}

\begin{proof}
    By Hochster's formula (\Cref{prop1})
    {\footnotesize
    \[
\beta_{i,i+d}(R_{G_1*G_2}/I_r(G_1*G_2))= \sum_{\underset{|W|=i+d}{W \subseteq V(G_1*G_2) }} \dim_{\mathbb{K}} \widetilde{H}_{d-1}(\Gamma[W]; \mathbb{K}).
\]
}
\noindent
 Now $V(G_1*G_2)=V(G_1)\sqcup V(G_2)$. Thus if $W\subseteq V(G_1)\sqcup V(G_2)$ with $|W|=i+d$, then $W=W_1\sqcup W_2$, where $W_1=W\cap V(G_1)$ and $W_2=W\cap V(G_2)$. Therefore, $0\le |W_1|\le i+d$ and $0\le |W_2|\le i+d$. Let $|W_1|=i+d$. Then $|W_2|=0$ and $W=W_1$. In that case using \Cref{subcomplex on subset}, we have $\Gamma[W]=\Gamma_1[W]=\Gamma_{11}[W_1]$. Hence
 {\footnotesize
 \begin{equation}\label{eqnew1}
 \sum_{\underset{|W_1|=i+d}{W=W_1\sqcup W_2 \subseteq V(G_1*G_2) }} \dim_{\mathbb{K}} \widetilde{H}_{d-1}(\Gamma[W]; \mathbb{K})=\sum_{\underset{|W_1|=i+d}{W_1 \subseteq V(G_1) }} \dim_{\mathbb{K}} \widetilde{H}_{d-1}(\Gamma_{11}[W_1]; \mathbb{K})=\beta_{i,i+d}(R_{G_1}/I_r(G_1)).
  \end{equation}}
 \noindent
 Similarly, taking $|W_2|=i+d$, we have 
 {\footnotesize
 \begin{equation}\label{eqnew2}
  \sum_{\underset{|W_2|=i+d}{W=W_1\sqcup W_2 \subseteq V(G_1*G_2) }} \dim_{\mathbb{K}} \widetilde{H}_{d-1}(\Gamma[W]; \mathbb{K})=\beta_{i,i+d}(R_{G_2}/I_r(G_2)).
 \end{equation}}
 \noindent
Thus we may assume that $|W_1|=i+d-j$ and $|W_2|=j$, where $1\le j\le i+d-1$. First, we consider the case $d=r$. In this case by \Cref{r-1 homology lemma},

{\footnotesize
    \begin{equation}\label{eqnew3}
    \begin{split}
     &\dim_{\K}(\widetilde H_{r-1}(\Gamma[W];\K))\\
     &=\dim_{\K}(\widetilde H_{r-1}(\Gamma_{11}[W_1];\K))+\dim_{\K}(\widetilde H_{r-1}(\Gamma_{22}[W_2];\K))+{i+r-1\choose r}-{i+r-1-j\choose r}-{j-1\choose r}.   
    \end{split}
    \end{equation}
    }
    Therefore,
    {\footnotesize
    \begin{align*}
       &\beta_{i,i+r}(R_{G_1*G_2}/I_r(G_1*G_2))\\
       &=\sum_{\underset{0\le j\le i+r}{\underset{|W_1|=i+r-j,\,|W_2|=j}{W=W_1\sqcup W_2 \subseteq V(G_1*G_2) }}} \dim_{\mathbb{K}} \widetilde{H}_{r-1}(\Gamma[W]; \mathbb{K})\\
       &=\sum_{\underset{|W_1|=i+r}{W_1 \subseteq V(G_1) }} \dim_{\mathbb{K}} \widetilde{H}_{r-1}(\Gamma_{11}[W_1]; \mathbb{K})+\sum_{\underset{|W_2|=i+r}{W_2 \subseteq V(G_2) }} \dim_{\mathbb{K}} \widetilde{H}_{r-1}(\Gamma_{22}[W_2]; \mathbb{K})\\
       &\hspace{18em}+\sum_{\underset{1\le j\le i+r-1}{\underset{|W_1|=i+r-j,\,|W_2|=j}{W=W_1\sqcup W_2 \subseteq V(G_1*G_2) }}} \dim_{\mathbb{K}} \widetilde{H}_{r-1}(\Gamma[W]; \mathbb{K})
       \end{align*}}
Now using \Cref{eqnew1}, \Cref{eqnew2} and \Cref{eqnew3} we have,

       {\footnotesize
    \begin{align*}
    &\beta_{i,i+r}(R_{G_1*G_2}/I_r(G_1*G_2))\\
       &=\beta_{i,i+r}(R_{G_1}/I_r(G_1))+\beta_{i,i+r}(R_{G_2}/I_r(G_2))\\
       &\hspace{6em}+\sum_{\underset{1\le j\le i+r-1}{\underset{|W_1|=i+r-j,\,|W_2|=j}{W=W_1\sqcup W_2 \subseteq V(G_1*G_2) }}}\left[\dim_{\K}(\widetilde H_{r-1}(\Gamma_{11}[W_1];\K))+\dim_{\K}(\widetilde H_{r-1}(\Gamma_{22}[W_2];\K))\right]\\
       &\hspace{12em}+\sum_{\underset{1\le j\le i+r-1}{\underset{|W_1|=i+r-j,\,|W_2|=j}{W=W_1\sqcup W_2 \subseteq V(G_1*G_2) }}}\left[{i+r-1\choose r}-{i+r-1-j\choose r}-{j-1\choose r}\right]\\
       &=\beta_{i,i+r}(R_{G_1}/I_r(G_1))+\beta_{i,i+r}(R_{G_2}/I_r(G_2))\\
       &\hspace{4em}+\sum_{j=1}^{i+r-1}\left[ {|V(G_2)|\choose j}\beta_{i-j,i-j+r}(R_{G_1}/I_r(G_1))+{|V(G_1)|\choose i-j+r}\beta_{j-r,j}(R_{G_2}/I_r(G_2)) \right]\\
       &\hspace{8em}+\sum_{j=1}^{i+r-1}\left[{i+r-1\choose r}-{i+r-1-j\choose r}-{j-1\choose r}\right]{|V(G_2)|\choose j}{|V(G_1)|\choose i+r-j}.
    \end{align*}}
    \noindent
    Now $\sum_{j=1}^{i+r-1}{|V(G_1)|\choose i-j+r}\beta_{j-r,j}(R_{G_2}/I_r(G_2))=\sum_{j=1}^{i+r-1}{|V(G_1)|\choose j}\beta_{i-j,i-j+r}(R_{G_2}/I_r(G_2))$. Thus
    {\footnotesize
    \begin{align*}
        &\beta_{i,i+r}(R_{G_1*G_2}/I_r(G_1*G_2))\\
        &=\sum_{j=0}^{i+r-1}\left[ {m\choose j}\beta_{i-j,i-j+r}(R_{G_1}/I_r(G_1))+{n\choose j}\beta_{i-j,i-j+r}(R_{G_2}/I_r(G_2)) \right]\\
       &\hspace{6em}+\sum_{j=1}^{i+r-1}\left[{i+r-1\choose r}-{i+r-1-j\choose r}-{j-1\choose r}\right]{m\choose j}{n\choose i-j+r}.
    \end{align*}}
    \noindent
    We now consider the case $d\ge r+1$. In this case by \Cref{i homology lemma},
    {\footnotesize
    \begin{align*}
     \dim_{\K}(\widetilde H_{d-1}(\Gamma[W];\K))=\dim_{\K}(\widetilde H_{d-1}(\Gamma_{11}[W_1];\K))+\dim_{\K}(\widetilde H_{d-1}(\Gamma_{22}[W_2];\K)).
    \end{align*}}
\noindent
    Therefore,
    {\footnotesize
    \begin{align*}
       &\beta_{i,i+d}(R_{G_1*G_2}/I_r(G_1*G_2))\\
       &=\sum_{\underset{0\le j\le i+d}{\underset{|W_1|=i+d-j,\,|W_2|=j}{W=W_1\sqcup W_2 \subseteq V(G_1*G_2) }}} \dim_{\mathbb{K}} \widetilde{H}_{d-1}(\Gamma[W]; \mathbb{K})\\
       &=\sum_{\underset{|W_1|=i+d}{W_1 \subseteq V(G_1) }} \dim_{\mathbb{K}} \widetilde{H}_{d-1}(\Gamma_{11}[W_1]; \mathbb{K})+\sum_{\underset{|W_2|=i+d}{W_2 \subseteq V(G_2) }} \dim_{\mathbb{K}} \widetilde{H}_{d-1}(\Gamma_{22}[W_2]; \mathbb{K})\\
       &\hspace{18em}+\sum_{\underset{1\le j\le i+d-1}{\underset{|W_1|=i+d-j,\,|W_2|=j}{W=W_1\sqcup W_2 \subseteq V(G_1*G_2) }}} \dim_{\mathbb{K}} \widetilde{H}_{d-1}(\Gamma[W]; \mathbb{K})
       \end{align*}
       }

\noindent
Note that $\sum_{\underset{|W_1|=i+d}{W_1 \subseteq V(G_1) }} \dim_{\mathbb{K}} \widetilde{H}_{d-1}(\Gamma_{11}[W_1]; \mathbb{K})=\beta_{i,i+d}(R_{G_1}/I_r(G_1))$ and similarly we have $\sum_{\underset{|W_2|=i+d}{W_2 \subseteq V(G_2) }} \dim_{\mathbb{K}} \widetilde{H}_{d-1}(\Gamma_{22}[W_2]; \mathbb{K})=\beta_{i,i+d}(R_{G_2}/I_r(G_2))$. Hence

       {\footnotesize
       \begin{align*}
       &\beta_{i,i+d}(R_{G_1*G_2}/I_r(G_1*G_2))\\
       &=\beta_{i,i+d}(R_{G_1}/I_r(G_1))+\beta_{i,i+d}(R_{G_2}/I_r(G_2))\\
       &\hspace{6em}+\sum_{\underset{1\le j\le i+d-1}{\underset{|W_1|=i+d-j,\,|W_2|=j}{W=W_1\sqcup W_2 \subseteq V(G_1*G_2) }}}\left[\dim_{\K}(\widetilde H_{d-1}(\Gamma_{11}[W_1];\K))+\dim_{\K}(\widetilde H_{d-1}(\Gamma_{22}[W_2];\K))\right]\\
       &=\beta_{i,i+d}(R_{G_1}/I_r(G_1))+\beta_{i,i+d}(R_{G_2}/I_r(G_2))\\
       &\hspace{6em}+\sum_{j=1}^{i+d-1}\left[ {|V(G_2)|\choose j}\beta_{i-j,i-j+d}(R_{G_1}/I_r(G_1))+{|V(G_1)|\choose i-j+d}\beta_{j-d,j}(R_{G_2}/I_r(G_2)) \right]\\
       &=\sum_{j=0}^{i+d-1}\left[{m\choose j}\beta_{i-j,i-j+d}(R_{G_1}/I_r(G_1))+{n\choose j}\beta_{i-j,i-j+d}(R_{G_2}/I_r(G_2)) \right].
       \end{align*}}

       This completes the proof.
\end{proof}

\begin{remark}
    The formula for the Betti numbers in \Cref{main theorem} can also be deduced using \cite[Corollary 5.3]{Hugh}. Indeed, if one takes $I=\langle x_1,\ldots,x_n\rangle$ and $J=\langle y_1,\ldots,y_m\rangle$, then the formula in \Cref{main theorem} follows after simplifying the formula in \cite[Corollary 5.3]{Hugh}.
\end{remark}

As a corollary of the above theorem, we can prove the following result on the regularity of $R_{G_1*G_2}/I_r(G_1*G_2)$.

\begin{corollary}\label{main corollary}
    Let $G_1$ and $G_2$ be two graphs on disjoint vertex sets and let $r$ be a positive integer. Then 
    \[
    \mathrm{reg}(R_{G_1*G_2}/I_r(G_1*G_2))=\max\left\{\mathrm{reg}(R_{G_1}/I_r(G_1)),\mathrm{reg}(R_{G_2}/I_r(G_2))\right\}
    \]
\end{corollary}

\begin{proof}
    Without loss of generality let $s=\mathrm{reg}(R_{G_1}/I_r(G_1))\ge \mathrm{reg}(R_{G_2}/I_r(G_2))$. Then we have $\beta_{i,i+s}(R_{G_1}/I_r(G_1))\neq 0$ for some $i$ and $\beta_{i,i+j}(R_{G_1}/I_r(G_1))=\beta_{i,i+j}(R_{G_2}/I_r(G_2))=0$ for all $i$ and for all $j\ge s+1$. Now we use the formula in \Cref{main theorem}. Note that, by induction on $d$, one can show that ${i+d-1\choose d}\ge {i+d-1-j\choose d}+{j-1\choose d}$. Hence $\beta_{i,i+s}(R_{G_1*G_2}/I_r(G_1*G_2))\neq 0$ for the above $i$ and $\beta_{i,i+j}(R_{G_1*G_2}/I_r(G_1*G_2))$ for all $i$ and for all $j\ge s+1$. Therefore, $\mathrm{reg}(R_{G_1*G_2}/I_r(G_1*G_2))=s=\max\left\{\mathrm{reg}(R_{G_1}/I_r(G_1)),\mathrm{reg}(R_{G_2}/I_r(G_2))\right\}$.
\end{proof}


\section{Betti numbers for some families of graphs}\label{Section 4}

In this section, we find explicit formulas for the Betti numbers and the regularity of the hyperedge ideals for some well-known families of graphs as an application of \Cref{main theorem}. First, we give the formula for complete graphs. Here we note that the formula of Betti numbers for complete graphs was first obtained in \cite[Theorem 4.4]{DRST} by showing that the ideal $I_r(K_n)$ is a vertex splittable ideal for all positive integer $r$. Moreover, the formula is well-known in the case of $r=1$ and it was first derived in the thesis of Jacques \cite[Theorem 5.1.1]{SJ}.

\begin{theorem}\label{complete graph}
    Let $K_n$ denote the complete graph on $n$ vertices. Then for all positive integers $r
    $ the $\N$-graded Betti numbers of the ideal $I_r(K_n)$ can be expressed as follows.
    {\footnotesize
    \[
\beta_{i,i+d}(R_{K_n}/I_r(K_n))=\begin{cases}
    {i+d-1\choose d}{n\choose i+d} &\text{ for } d=r \text{ and } 1\le i\le n-r,\\
    0 &\text{ otherwise.}
\end{cases}
\]  }
\end{theorem}

\begin{proof}
    We prove the formula by induction on $n$. Let $r$ be fixed. If $n\le r$ then $I_r(K_n)=0$. If $n=r+1$, then $I_r(K_n)=\langle \prod_{i=1}^{r+1}x_i \rangle$. Hence 
     {\footnotesize
     \[
    0\rightarrow R_{K_n}(-r-1)\xrightarrow{\cdot \prod_{i=1}^{r+1}x_i}R_{K_n}\rightarrow \frac{R_{K_n}}{\left\langle \prod_{i=1}^{r+1}x_i\right\rangle}\rightarrow 0
    \]
    }
    
    \noindent
    is a minimal free resolution of $R_{K_n}/I_r(K_n)$. Therefore, we have the above formula. Now let $n>r+1$. The complete graph $K_n$ can be expressed as $K_{n-1}*G$, where $G$ is the graph consisting of only one vertex. Then by induction we see that $\beta_{i,i+j}(R_{K_n}/I_r(K_n))=0$ if $j\neq r$. Moreover, by \Cref{main theorem}, we have
    {\footnotesize
    \begin{align*}
        &\beta_{i,i+r}(R_{K_n}/I_r(K_n))\\
        &=\sum_{j=0}^{i+r-1}\left\{{1\choose j}\beta_{i-j,i-j+r}(R_{K_{n-1}}/I_r(K_{n-1}))+{n-1\choose j}\beta_{i-j,i-j+r}(R_{G}/I_r(G)) \right\}+\\
&\hspace{7em}   \sum_{j=1}^{i+r-1} \left[{i+r-1\choose r}-{i+r-1-j\choose r}-{j-1\choose r}\right]{1\choose j}{n-1\choose i+r-j}\\
&=\beta_{i,i+r}(R_{K_{n-1}}/I_r(K_{n-1}))+\beta_{i-1,i+r-1}(R_{K_{n-1}}/I_r(K_{n-1}))+\\
&\hspace{7em} \left[{i+r-1\choose r}-{i+r-2\choose r}-{0\choose r}\right]{n-1\choose i+r-1}\\
&={i+r-1\choose r}{n\choose i+r}.
    \end{align*}}
    
\end{proof}

Since $\beta_{i,i+d}(R_{K_n}/I_r(K_n))= 0$ for $d\neq r$, we have the following corollary.
\begin{corollary}\label{corollary 1}
    For all positive integer $r$, $\mathrm{reg}(R_{K_n}/I_r(K_n))=r$.
\end{corollary}

\begin{example}
   Let $G=K_6$ and $r=3$. Then the Betti table of $R_{K_6}/I_3(K_6)$ computed using \Cref{complete graph} is given as follows.
    \begin{figure}[h!]\
\centering
\begin{tikzpicture}
[scale=.45]
\draw [fill] (4,2) circle [radius=0.1];
\draw [fill] (7,2) circle [radius=0.1];
\draw [fill] (9,4) circle [radius=0.1];
\draw [fill] (9,6) circle [radius=0.1];
\draw [fill] (2,4) circle [radius=0.1];
\draw [fill] (2,6) circle [radius=0.1];
\node at (4,1.2){$x_1$};
\node at (7,1.2){$x_2$};
\node at (10,4){$x_3$};
\node at (10,6){$x_4$};
\node at (1,6){$x_6$};
\node at (1,4){$x_7$};
\draw (4,2)--(7,2)--(9,4)--(2,4)--(2,6)--(4,2)--(9,4)--(2,4)--(7,2)--(9,6)--(2,6)--(4,2)--(9,6)--(2,4)--(4,2)--(2,4)--(7,2)--(2,6)--(9,4)--(9,6)--(9,4);

\draw [fill] (16,7) circle [radius=0.01];
\draw [fill] (16,8.5) circle [radius=0.01];
\draw (16,7)--(23.7,7);
\draw (16,8.5)--(23.7,8.5);
\node at (18,7.8){$0$};
\node at (19.4,7.8){$1$};
\node at (21.2,7.8){$2$};
\node at (23.3,7.8){$3$};
\node at (16.5,6.1){$0:$};
\node at (16.5,4.4){$1:$};
\node at (16.5,2.6){$2:$};
\node at (18,6){$1$};
\node at (19.4,6){$\cdot$};
\node at (21.2,6){$\cdot$};
\node at (23.3,6){$\cdot$};
\node at (18,4.3){$\cdot$};
\node at (19.4,4.3){$\cdot$};
\node at (21.2,4.3){$\cdot$};
\node at (23.3,4.3){$\cdot$};
\node at (18,2.5){$\cdot$};
\node at (19.4,2.5){$\cdot$};
\node at (21.2,2.5){$\cdot$};
\node at (23.3,2.5){$\cdot$};
\node at (16.5,0.8){$3:$};
\node at (18,0.7){$\cdot$};
\node at (19.4,0.7){$15$};
\node at (21.2,0.7){$24$};
\node at (23.3,0.7){$10$};
\end{tikzpicture}
\caption{ The graph $K_6$ and the Betti table of $R_{K_6}/I_3(K_6)$}
\end{figure}
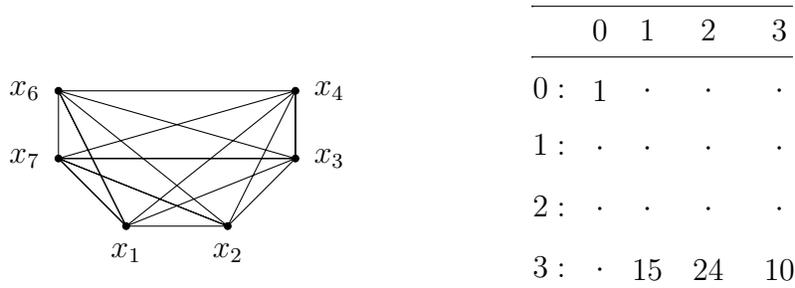
 \end{example}

The complement of a multipartite graph is a chordal graph. Hence the edge ideal of $\C$ when $G$ is a multipartite graph, is a vertex splittable ideal (by \cite[Theorem 3.12]{DRST}). In the next theorem, we provide explicit formulas for all the Betti numbers for the complete multipartite graphs. The formula is well-known in the case of $r=1$ and it was first derived in the thesis of Jacques \cite[Theorem 5.3.8]{SJ}.

\begin{theorem}\label{bipartite}
    Let $K_{n_1,n_2,\ldots,n_t}$ denote the complete multipartite graph on $\sum_{s=1}^tn_s$ number of vertices. Then for all positive integer $r$ the $\N$-graded Betti numbers of the ideal $I_r(K_{n_1,n_2,\ldots,n_t})$ can be expressed as follows.
    {\footnotesize
    \[
\beta_{i,i+d}(R_{K_{n_1,n_2,\ldots,n_t}}/I_r(K_{n_1,n_2,\ldots,n_t}))=\begin{cases}
    {\sum_{s=1}^tn_s\choose i+r}{i+r-1\choose r}-\sum_{\underset{\sum_{s=1}^tj_s=i+r}{(j_1,j_2,\ldots,j_t)\in\mathbb N^t}} \left[\prod_{s=1}^t{n_s\choose j_s}\right]\sum_{s=1}^t{j_s-1\choose r}&\text{ for }d=r,\\
    0 &\text{ otherwise.}
\end{cases}
\]  }

\end{theorem}

\begin{proof}
    We prove this formula by induction on $t$. First, consider the case when $t=2$. Let $G_1$ denote the graph on $n_1$ vertices with no edges and let $G_2$ denote the graph on $n_2$ vertices with no edges. Then $K_{n_1,n_2}=G_1*G_2$. By the formula in \Cref{main theorem} we see that $\beta_{i,i+d}(R_{K_{n_1,n_2}}/I_r(K_{n_1,n_2}))=0$ if $d\neq r$. Moreover,
{\footnotesize
\begin{align*}
\beta_{i,i+r}(R_{K_{n_1,n_2}}/I_r(K_{n_1,n_2}))&=\sum_{j=0}^{i+r-1}\left\{{n_2\choose j}\beta_{i-j,i-j+r}(R_{G_1}/I_r(G_1))+{n_1\choose j}\beta_{i-j,i-j+r}(R_{G_2}/I_r(G_2)) \right\}+\\
&\hspace{5em}   \sum_{j=1}^{i+r-1} \left[{i+r-1\choose r}-{i+r-1-j\choose r}-{j-1\choose r}\right]{n_2\choose j}{n_1\choose i+r-j}\\ 
&=\sum_{j=1}^{i+r-1} \left[{i+r-1\choose r}-{i+r-1-j\choose r}-{j-1\choose r}\right]{n_2\choose j}{n_1\choose i+r-j}\\
&=\sum_{j=1}^{i+r-1}{i+r-1\choose r}{n_2\choose j}{n_1\choose i+r-j}-\sum_{j=1}^{i-1}{i+r-1-j\choose r}{n_2\choose j}{n_1\choose i+r-j}-\\
&\hspace{23em}\sum_{j=r+1}^{i+r-1}{j-1\choose r}{n_2\choose j}{n_1\choose i+r-j}
\end{align*}
}
Now using the formula $\sum_{j=0}^t{a\choose j}{b\choose t-j}={a+b\choose t}$ we get the following.

{\footnotesize
\begin{align*}
    &\beta_{i,i+r}(R_{K_{n_1,n_2}}/I_r(K_{n_1,n_2}))\\
    &={n_1+n_2\choose i+r}{i+r-1\choose r}-\sum_{\underset{j_1+j_2=i+r}{(j_1,j_2)\in\mathbb N^2}}{n_1\choose j_1}{n_2\choose j_2}{j_1-1\choose r}-\sum_{\underset{j_1+j_2=i+r}{(j_1,j_2)\in\mathbb N^2}}{n_1\choose j_1}{n_2\choose j_2}{j_2-1\choose r}\\
    &={n_1+n_2\choose i+r}{i+r-1\choose r}-\sum_{\underset{j_1+j_2=i+r}{(j_1,j_2)\in\mathbb N^2}}\left[ 
\prod_{s=1}^2{n_s\choose j_s} \right]\sum_{s=1}^2{j_s-1\choose r}.
\end{align*}
}

Now suppose $t>2$ and the formula is true for all positive integers less than $t$. Let $G_1'=K_{n_1,n_2,\ldots,n_{t-1}}$ and let $G_2'$ denote the graph on $n_t$ vertices with no edges. Then $K_{n_1,n_2,\ldots,n_t}=G_1'*G_2'$. Now applying the formula in \Cref{main theorem} again and using induction hypothesis we get $\beta_{i,i+d}(R_{K_{n_1,n_2,\ldots,n_t}}/I_r(K_{n_1,n_2,\ldots,n_t}))=0$ if $d\neq r$. Moreover,

{\footnotesize
\begin{align*}
    &\beta_{i,i+r}(R_{K_{n_1,n_2,\ldots,n_t}}/I_r(K_{n_1,n_2,\ldots,n_t}))\\
    &=\sum_{j=0}^{i+r-1}\left\{{n_t\choose j}\beta_{i-j,i-j+r}(R_{G_1'}/I_r(G_1'))+{\sum_{s=1}^{t-1}n_s\choose j}\beta_{i-j,i-j+r}(R_{G_2'}/I_r(G_2')) \right\}+\\
&\hspace{15em}   \sum_{j=1}^{i+r-1} \left[{i+r-1\choose r}-{i+r-1-j\choose r}-{j-1\choose r}\right]{n_t\choose j}{\sum_{s=1}^{t-1}n_s\choose i+r-j}\\ 
\end{align*}
}

Note that $\beta_{i-j,i-j+r}(R_{G_2'}/I_r(G_2'))=0$ for all $r\ge 1$. Hence 

{\footnotesize
\begin{align*}
&\beta_{i,i+r}(R_{K_{n_1,n_2,\ldots,n_t}}/I_r(K_{n_1,n_2,\ldots,n_t}))\\
&=\sum_{j=0}^{i+r-1}{n_t\choose j}\left[ {\sum_{s=1}^{t-1}n_s\choose i-j+r}{i-j+r-1\choose r}-\sum_{\underset{\sum_{s=1}^{t-1}j_s=i-j+r}{(j_1,j_2,\ldots,j_{t-1})\in\mathbb N^{t-1}}} \left\{\prod_{s=1}^{t-1}{n_s\choose j_s}\right\}\sum_{s=1}^{t-1}{j_s-1\choose r} \right]+\\
&\hspace{8em} {i+r-1\choose r}\sum_{j=1}^{i+r-1}{n_t\choose j}{\sum_{s=1}^{t-1}n_s\choose i-j+r}-\sum_{j=1}^{i+r-1}{i-j+r-1\choose r}{n_t\choose j}{\sum_{s=1}^{t-1}n_s\choose i-j+r}-\\
&\hspace{31em}\sum_{j=1}^{i+r-1}{j-1\choose r}{n_t\choose j}{\sum_{s=1}^{t-1}n_s\choose i-j+r}\\
&={i+r-1\choose r}\sum_{j=0}^{i+r}{n_t\choose j}{\sum_{s=1}^{t-1}n_s\choose i-j+r}-{n_t\choose i+r}{i+r-1\choose r}-\\
&\hspace{6.5em}\sum_{j=0}^{i+r-1}{n_t\choose j}\sum_{\underset{\sum_{s=1}^{t-1}j_s=i-j+r}{(j_1,j_2,\ldots,j_{t-1})\in\mathbb N^{t-1}}} \left\{\prod_{s=1}^{t-1}{n_s\choose j_s}\right\}\sum_{s=1}^{t-1}{j_s-1\choose r}-\sum_{j=1}^{i+r-1}{n_t\choose j}{j-1\choose r}{\sum_{s=1}^{t-1}n_s\choose i-j+r}
\end{align*}
}
Note that $\sum_{j=0}^{i+r}{n_t\choose j}{\sum_{s=1}^{t-1}n_s\choose i-j+r}={\sum_{s=1}^{t}n_s\choose i-j+r}$. Thus
{\footnotesize
\begin{align*}
&\beta_{i,i+r}(R_{K_{n_1,n_2,\ldots,n_t}}/I_r(K_{n_1,n_2,\ldots,n_t}))\\
&={i+r-1\choose r}{\sum_{s=1}^{t}n_s\choose i+r}-\sum_{j=0}^{i+r-1}{n_t\choose j}\sum_{\underset{\sum_{s=1}^{t-1}j_s=i-j+r}{(j_1,j_2,\ldots,j_{t-1})\in\mathbb N^{t-1}}} \left\{\prod_{s=1}^{t-1}{n_s\choose j_s}\right\}\sum_{s=1}^{t-1}{j_s-1\choose r}-\\
&\hspace{33em} \sum_{j=1}^{i+r}{n_t\choose j}{j-1\choose r}{\sum_{s=1}^{t-1}n_s\choose i-j+r}\\
\end{align*}
}
Now by induction on $s$ we have ${\sum_{s=1}^{t-1}n_s\choose i-j+r}=\sum_{\underset{\sum_{s=1}^{t-1}j_s=i-j+r}{(j_1,j_2,\ldots,j_{t-1})\in\mathbb N^{t-1}}}\prod_{s=1}^{t-1}{n_s\choose j_s}$. Hence
{\footnotesize
\begin{align*}
&\beta_{i,i+r}(R_{K_{n_1,n_2,\ldots,n_t}}/I_r(K_{n_1,n_2,\ldots,n_t}))\\
&={i+r-1\choose r}{\sum_{s=1}^{t}n_s\choose i+r}- \sum_{j=0}^{i+r}{n_t\choose j}\left[\sum_{\underset{\sum_{s=1}^{t-1}j_s=i-j+r}{(j_1,j_2,\ldots,j_{t-1})\in\mathbb N^{t-1}}} \left\{\prod_{s=1}^{t-1}{n_s\choose j_s}\right\}\left\{\sum_{s=1}^{t-1}{j_s-1\choose r}+{j-1\choose r}\right\}\right]\\
&={i+r-1\choose r}{\sum_{s=1}^{t}n_s\choose i+r}-\sum_{\underset{\sum_{s=1}^{t}j_s=i+r}{(j_1,j_2,\ldots,j_{t})\in\mathbb N^{t}}} \left\{\prod_{s=1}^{t}{n_s\choose j_s}\right\}\sum_{s=1}^{t}{j_s-1\choose r}.
\end{align*}
}
This completes the proof.
\end{proof}

Again, since $\beta_{i,i+d}(R_{K_{n_1,n_2,\ldots,n_t}}/I_r(K_{n_1,n_2,\ldots,n_t}))= 0$ for $d\neq r$, we have the following corollary.
\begin{corollary}\label{corollary 2}
    For all positive integer $r$, $\mathrm{reg}(R_{K_{n_1,n_2,\ldots,n_t}}/I_r(K_{n_1,n_2,\ldots,n_t}))=r$.
\end{corollary}

\begin{remark}
    The formula for Betti numbers in \Cref{bipartite} can also be computed using the Hilbert series of the ideal $I_r(K_{n_1,n_2,\ldots,n_t})$ (see \cite[Theorem 4.9]{DRST}) and some combinatorial argument on the number of faces of the $r$-independence complex of the graph. Here we provide an alternative way to find the formula.
\end{remark}

\begin{example}
Let $G=K_{6,6}$ and $r=3$. Then the Betti table of $R_{K_{6,6}}/I_3(K_{6,6})$ computed using \Cref{bipartite} is given as follows.
\begin{figure}[h!]
\centering
\begin{tikzpicture}
[scale=.45]
    \draw [fill] (2,0) circle [radius=0.1];
\draw [fill] (2,1) circle [radius=0.1];
\draw [fill] (2,2) circle [radius=0.1];
\draw [fill] (2,3) circle [radius=0.1];
\draw [fill] (2,4) circle [radius=0.1];
\draw [fill] (2,5) circle [radius=0.1];
\draw [fill] (11,0) circle [radius=0.1];
\draw [fill] (11,1) circle [radius=0.1];
\draw [fill] (11,2) circle [radius=0.1];
\draw [fill] (11,3) circle [radius=0.1];
\draw [fill] (11,4) circle [radius=0.1];
\draw [fill] (11,5) circle [radius=0.1];
\draw (2,0)--(11,0)--(2,1)--(11,1)--(2,2)--(11,2)--(2,3)--(11,3)--(2,4)--(11,4)--(2,5)--(11,5)--(2,4)--(11,2)--(2,1)--(11,3)--(2,5)--(11,2)--(2,0)--(11,3)--(2,2)--(11,0)--(2,3)--(11,1)--(2,0)--(11,4)--(2,1)--(11,5)--(2,3)--(11,4)--(2,2)--(11,5)--(2,0);
\draw (11,0)--(2,4)--(11,1)--(2,5)--(11,0);
\node at (1.3,0) {$x_1$};
\node at (1.3,1) {$x_2$};
\node at (1.3,2) {$x_3$};
\node at (1.3,3) {$x_4$};
\node at (1.3,4) {$x_5$};
\node at (1.3,5) {$x_6$};
\node at (11.7,0) {$y_1$};
\node at (11.7,1) {$y_2$};
\node at (11.7,2) {$y_3$};
\node at (11.7,3) {$y_4$};
\node at (11.7,4) {$y_5$};
\node at (11.7,5) {$y_6$};

\draw [fill] (37.6,5) circle [radius=0.01];
\draw [fill] (13.8,5) circle [radius=0.01];
\draw [fill] (37.6,6.5) circle [radius=0.01];
\draw [fill] (13.8,6.5) circle [radius=0.01];
\draw (13.8,5)--(37.6,5);
\draw (13.8,6.5)--(37.6,6.5);
\node at (16,5.8){$0$};
\node at (17.4,5.8){$1$};
\node at (19.5,5.8){$2$};
\node at (21.9,5.8){$3$};
\node at (24.5,5.8){$4$};
\node at (27.5,5.8){$5$};
\node at (30.3,5.8){$6$};
\node at (32.9,5.8){$7$};
\node at (35.1,5.8){$8$};
\node at (37,5.8){$9$};
\node at (14.5,4.3){$0:$};
\node at (14.5,2.5){$1:$};
\node at (16,4.2){$1$};
\node at (17.4,4.2){$\cdot$};
\node at (19.5,4.2){$\cdot$};
\node at (21.9,4.2){$\cdot$};
\node at (24.5,4.2){$\cdot$};
\node at (27.5,4.2){$\cdot$};
\node at (30.3,4.2){$\cdot$};
\node at (32.9,4.2){$\cdot$};
\node at (35.1,4.2){$\cdot$};
\node at (37,4.2){$\cdot$};
\node at (16,2.4){$\cdot$};
\node at (17.4,2.4){$\cdot$};
\node at (19.5,2.4){$\cdot$};
\node at (21.9,2.4){$\cdot$};
\node at (24.5,2.4){$\cdot$};
\node at (27.5,2.4){$\cdot$};
\node at (30.3,2.4){$\cdot$};
\node at (32.9,2.4){$\cdot$};
\node at (35.1,2.4){$\cdot$};
\node at (37,2.4){$\cdot$};
\node at (14.5,0.7){$2:$};
\node at (16,0.6){$\cdot$};
\node at (17.4,0.6){$\cdot$};
\node at (19.5,0.6){$\cdot$};
\node at (21.9,0.6){$\cdot$};
\node at (24.5,0.6){$\cdot$};
\node at (27.5,0.6){$\cdot$};
\node at (30.3,0.6){$\cdot$};
\node at (32.9,0.6){$\cdot$};
\node at (35.1,0.6){$\cdot$};
\node at (37,0.6){$\cdot$};

\node at (14.5,-1.1){$3:$};
\node at (16,-1.1){$\cdot$};
\node at (17.4,-1.1){$465$};
\node at (19.5,-1.1){$2940$};
\node at (21.9,-1.1){$8482$};
\node at (24.5,-1.1){$14400$};
\node at (27.5,-1.1){$15615$};
\node at (30.3,-1.1){$11020$};
\node at (32.9,-1.1){$4926$};
\node at (35.1,-1.1){$1272$};
\node at (37,-1.1){$145$};

\end{tikzpicture}
\caption{ The graph $K_{6,6}$ and the Betti table of $R_{K_{6,6}}/I_3(K_{6,6})$}
\end{figure}
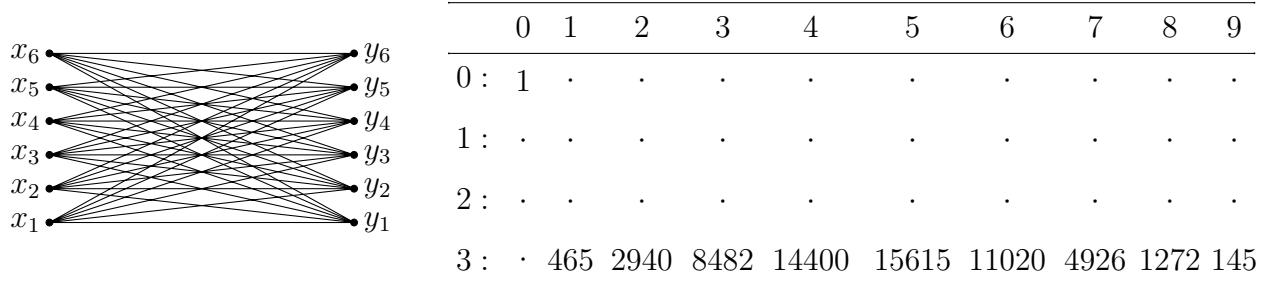
\end{example}

As an application of \Cref{bipartite} we get formulas for $\mathbb N$-graded Betti numbers of the Stanley-Reisner ideal of $r$-independence complexes of star graphs.  For $r=1$ case, this formula was also first derived in the thesis of Jacques \cite[Theorem 5.4.11]{SJ}.  

\begin{corollary}\label{star graph case}
    Let $K_{1,n}$ denote the star graph on $n+1$ vertices, where $V(K_{1,n})=\{x_0,x_1,\dots, x_n\}$ and $E(K_{1,n})=\{\{x_0,x_i\}\mid i\in [n]\}$. Then the $\mathbb N$-graded Betti numbers of $R_{K_{1,n}}/I_{r}(K_{1,n})$, for $r\ge 1$, can be expressed as follows.
    \[
    \beta_{i,i+d}(R_{K_{1,n}}/I_r(K_{1,n}))=\begin{cases}
        {i+r-2\choose r-1}{n\choose i+r-1} &\text{ for } d=r \text{ and }1\le i\le n-r+1,\\
        0 &\text{ otherwise. }
    \end{cases}
    \]
\end{corollary}

Let $G$ be a finite simple graph on the vertex set $\{x_1,\ldots,x_n\}$ and let $t$ be an integer such that $t>1$. A path of length $t$ from a vertex $u$ in $G$ to another vertex $v$ is a sequence of vertices $u=x_{i_1},x_{i_2},\ldots,x_{i_t}=v$ such that $\{x_{i_j},x_{i_{j+1}}\}\in E(G)$ for each $j\in [t-1]$. The $t$-path ideal of $G$, denoted by $J_t(G)$, is a monomial ideal generated by the monomials $\left\{\prod_{j=1}^tx_{i_j}\mid \{x_{i_1},x_{i_2},\ldots,x_{i_t}\} \text{ is a path in }G \right\}$.

Note that if $C_n$ is a cycle of length $n$ then $J_{r+1}(C_n)=I_r(C_n)$. The graded Betti numbers of the path ideal of $C_n$ are calculated by Alilooee and Faridi in \cite{AlFa}. Thus using \cite[Theorem 4.13]{AlFa} we can express the graded Betti numbers of $I_r(C_n)$ as follows.

\begin{theorem}\label{cycle graph}
    Let $n,r,p$ and $q$ be integers such that $1\le r\le n-1$ and $n=(r+2)p+q$, where $p\ge 0$ and $0\le q\le r+1$. Then
    \[
    \beta_{i,i+d}(R_{C_{n}}/I_r(C_{n}))=\begin{cases}
        r+1 &\text{ if }i=2p, d=n-2p, q=0,\\
        1&\text{ if }i=2p+1, d=n-2p-1, q\neq 0,\\
        \frac{n}{n-\frac{d(r+1)}{r}}{\frac{d}{r}\choose i-\frac{d}{r}}{n-\frac{d(r+1)}{r}\choose \frac{d}{r}}&\text{ if }i<n-d,0\le \frac{d}{r}\le i,i\le \frac{2d}{r}\le 2p,\\
        0&\text{ otherwise.}
    \end{cases}
    \]
\end{theorem}

Now we proceed to determine the graded Betti numbers of $I_r(W_{n+1})$, where $W_{n+1}$ is the wheel graph on $n+1$ vertices. Let $H$ be the graph consisting of only one vertex. Then $W_{n+1}=C_n*H$. Therefore, using \Cref{main theorem}, \Cref{cycle graph} and the Pascal's identity we get the following.

\begin{theorem}\label{wheel graph}
    The $\mathbb N$-graded Betti numbers of the wheel graph $W_{n+1}$ can be expressed as follows.
    \begin{align*}
    &\beta_{i,i+d}(R_{W_{n+1}}/I_r(W_{n+1}))\\
    &=\begin{cases}
        \beta_{i,i+d}(R_{C_n}/I_r(C_{n}))+\beta_{i-1,i+d-1}(R_{C_n}/I_r(C_{n}))+{i+d-2\choose d-1}{n\choose i+d-1} &\text{ for }d=r,\\
        \beta_{i,i+d}(R_{C_n}/I_r(C_{n}))+\beta_{i-1,i+d-1}(R_{C_n}/I_r(C_{n})) &\text{ for }d\ge r+1.
    \end{cases}
    \end{align*}
\end{theorem}

Using \Cref{main corollary} we get the following.

\begin{corollary}\label{corollary 3}
    Let $n,r,p$ and $q$ be integers such that $1\le r\le n-1$ and $n=(r+2)p+q$, where $p\ge 0$ and $0\le q\le r+1$. Then
    \[
    \mathrm{reg}(R_{W_{n+1}}/I_r(W_{n+1}))=\mathrm{reg}(R_{C_n}/I_r(C_n))=\begin{cases}
        rp+q-1&\text{  if  }q\neq 0,\\
        rp&\text{  if  }q=0.
    \end{cases}
    \]
\end{corollary}

\section*{Acknowledgements} The authors would like to thank Priyavrat Deshpande, Anurag Singh, Adam Van Tuyl, and the anonymous referee for some helpful comments and suggestions. Amit Roy acknowledges the financial support from the Department of Atomic Energy, Government of India through a postdoctoral research fellowship during his stay in NISER Bhubaneswar. He is also partially supported by a grant from the Infosys Foundation.

\subsection*{Data availability statement} Data sharing is not applicable to this article as no new data were created or
analyzed in this study.

\section{Statements and Declarations}
\subsection*{Funding} The authors declare that no funds, grants, or other support were received during the preparation of this manuscript.

\subsection*{Competing Interests} The authors have no relevant financial or non-financial interests to disclose.

\subsection*{Author Contributions} All authors contributed equally.

\end{document}